\documentclass[reqno]{amsart}
\usepackage{amsmath,latexsym,amssymb,amsfonts,amsthm,subfig,graphicx,color,ulem, mathtools}
\usepackage{tikz}
\usetikzlibrary{arrows,calc}

\newcommand{\eqdef}{\!\overset{\text{\tiny def}}{=}\!}
\newcommand{\R}{\mathbb{R}}
\newcommand{\llangle}{\langle\!\langle}
\newcommand{\rrangle}{\rangle\!\rangle}

\newcommand{\E}{\mathbf{E}}
\newcommand{\rate}{\Lambda}

\newcommand{\Hfull}{H_0}
\newcommand{\Lfull}{\mathcal{L}_0}

\newcommand{\Hharm}{\bar{H}}
\newcommand{\hf}{\frac1{2}}
\newcommand{\betaI}{\tfrac{1}{\beta}}
\newcommand{\alphaI}{\tfrac{1}{\alpha}}
\newcommand{\alphaII}{\tfrac{2}{\alpha}}
\renewcommand{\H}{\mathbb{H}}

\newcommand{\Qcurve}{\varrho}
\newtheorem{theorem}{Theorem}[section]
\newtheorem{corollary}[theorem]{Corollary}
\newtheorem{lemma}[theorem]{Lemma}

\newtheorem{remark}[theorem]{Remark}
\newtheorem{definition}[theorem]{Definition}

\usepackage{color,cancel}

\renewcommand{\H}{\mathbb{H}}

\usepackage{color,cancel}

\title[Ergodicity of systems with
  a Lennard--Jones--like Potential]{Geometric  Ergodicity of Two--dimensional Hamiltonian  systems with
  a Lennard--Jones--like Repulsive Potential}

\author[B. Cooke]{Ben Cooke}
\address{ Department of Mathematics, Duke University, Durham NC}
\author[D.P. Herzog]{David P. Herzog}
\address{Department of Mathematics, Iowa State University, Ames IA}
\author[J. C. Mattingly]{Jonathan C. Mattingly}
\address{Department of Mathematics, Center for Theoretical
    and Mathematical Sciences,
Center for Nonlinear and Complex Systems, and Department of Statistical
    Sciences\\ Duke
    University, Durham, NC, 27708-0251}
\author[S. A. McKinley]{Scott A. McKinley}
\address{ Department of Mathematics, Tulane University, New Orleans, LA}
\author[S. C. Schmidler]{Scott C. Schmidler}
\address{Department of Statistical
    Science, Department of Computer
    Science, Program in Computational Biology and Bioinformatics,
    Program in Structural Biology and Biophysics, Duke University,
    Durham NC 27708}
\date{\today}

\begin{document}
\maketitle

\newcommand{\p}{\mathbf{p}}
\newcommand{\q}{\mathbf{q}}
\newcommand{\abs}[1]{\left|#1\right|}
\newcommand{\me}{\mathrm{e}}

\section{Introduction}

Molecular dynamics simulation is among the most important and widely
used tools in the study of molecular systems, providing fundamental
insights into molecular mechanisms at a level of detail unattainable
by experimental methods
\cite{Allen:1987,Leach:1996,Frenkel:1996,Schlick:2002,Tuckerman_2010}.  Usage of
molecular dynamics spans a diverse array of fields, from physics and
chemistry, to molecular and cellular biology, to engineering and
materials science. Due to their size and complexity, simulations of
large systems such as biological macromolecules (DNA, RNA, proteins,
carbohydrates, and lipids) are typically performed under a classical
mechanics representation. A critical requirement of such simulations
is ergodicity, or convergence in the limit to the equilibrium
(typically canonical) Boltzmann measure $\mu(d\q,d\p) =
Z(\beta)^{-1}\me^{-\beta H(\q,\p)}d\q d\p$.  Although ergodicity is commonly
assumed,
recently \cite{Cooke:2008} showed that many commonly used
deterministic dynamics methods for simulating the canonical
(constant-temperature) ensemble fail to be ergodic.  They also showed
that introduction of a stochastic hybrid Monte Carlo (HMC) corrector
guarantees ergodicity; however, HMC scales poorly with system
dimension and is rarely used for macromolecules.  \cite{Cooke:2008}
also show empirically that more commonly used stochastic Langevin
dynamics \cite{Pastor:1994} appear to exhibit ergodic behavior, but
were unable to provide rigorous proof.

The key difficulty in applying existing arguments
\cite{MattinglyStuartHigham02} is the appearance of singularities in the
potential $U(\q)$. Most modern molecular mechanics force fields
\cite{Pearlman:1995,Brooks:1983,Jorgensen:1988} take the form
\begin{align*}
U(\q) & = \sum_{\rm bonds} K_1 (r - r^*)^2
+ \sum_{\rm angles} K_2 (\theta - \theta^*)^2 
+ \sum_{\rm dihedrals} \frac{V_n}{2} 
[1 + {\rm cos}(n\phi - \gamma)] \\
& \qquad + \;\; \sum_{i<j} \left [ \frac{A_{ij}}{r_{ij}^{12}} - 
  \frac{B_{ij}}{r_{ij}^6} + 
  \frac{q_iq_j}{\epsilon \, r_{ij}} 
  \right ]
\end{align*}
Here the first three terms involve bond length, angle, and torsional
energies; being bounded, these are easily handled. The difficulty
arises from the non-covalent electrostatic and Van der Waals forces,
the latter modeled by a Lennard-Jones potential, which give rise to
singularities as two atoms in the system approach each other at close
range.

In this paper we establish ergodicity of Langevin dynamics for
a simple two-particle system involving a Lennard-Jones type potential.
Moreover, we show that
the dynamics is {\it geometrically} ergodic (i.e. has a spectral gap) and
converges at a geometric rate.  Geometric ergodicity is sufficient to
imply existence of a central limit theorem for ergodic averages of
functions $f$ with $\E_\mu(\abs{f}^{2+\delta})< \infty$ for some $\delta
>0$ \cite{Ibragimov:1971}, and also implies the existence of an exact
sampling scheme \cite{Kendall:2004}, although the latter need not be
practical. Loosely, proving an ergodic result has two central
ingredients. One provides continuity of the transition densities in
total variation norm which ensures that transitions from nearby points
behave similarly enough probabilistically, providing the basic mechanism of the
probabilistic mixing/coupling. This is often expressed in a
minorization condition (see Lemma~\ref{smallset}).  The other
ingredient gives control of
excursions towards infinity which ensures the existence of a
stationary measure and guarantees that sufficient probabilistic
mixing for an exponential convergence rate. The difficulty
in a problem is typically one or the other. 

As this paper was being accepted for publication,
we became aware of two papers which prove results related to this
paper; namely, \cite{MR2674062,MR3413926}.   The results are different
in the cases where both apply. Here we prove
exponential convergence to equilibrium from arbitrary initial data in
variants of the total variation distance by building an optimal Lyapunov
function. Consequently, our methods can handle weighted norms whose weight functions grow faster at
infinity.  In \cite{MR2674062,MR3413926}, the convergence of time
averages is proven in $L^2$ when the system is started from
equilibrium.  In this sense, these results are together best characterized as mixing and make use in a critical way that the invariant measure is known as they build on the idea of Hypercoercivity.  
However, the scope of these two impressive papers,
\cite{MR2674062,MR3413926}, is much larger. For example, they are able to handle the chain of
interacting diffusions while we handle only two particles interacting
currently with our methods.

In Section~\ref{ModelProblem}, we will see that in the current
setting, basic existence of a stationary measure is trivial since the
standard Gibbs measure built from the energy is invariant. Uniqueness
of the stationary distribution follows from now standard results on
hypoelliptic diffusions. However the control necessary to give a
convergence rate or even convergence  has previously been elusive.  Our approach follows
the established method of demonstrating the existence of a Lyapunov
function and associated small set; however, construction of the
Lyapunov function in the presence of a singular potential is
non-trivial and our approach constitutes one of the major innovations
of this paper. In many ways it builds on ideas in 
\cite{HairerMattingly:2009} and more obliquely is related to the ideas
in \cite{Rey-BelletThomas:2002}. In both cases, time averaging of the
instantaneous energy dissipation rate is used to build a Lyapunov
function. We use similar ideas here. In a nutshell, as in
\cite{HairerMattingly:2009} the technique consists of casting the
behavior of the system as the energy heads to infinity as a problem
with order one energy containing a small parameter equal to one over
the original system's energy. Then, classical stochastic averaging
techniques are used to build a Lyapunov function. Though the solution
is related to \cite{HairerMattingly:2009}, the presentation of
difficulties is quite different. In particular, we will see that
extracting the asymptotic behavior is more difficult than
\cite{HairerMattingly:2009} as our potentials do not strictly scale
homogeneously. To overcome this we will use the idea of approximating
the dynamics near the point at infinity from \cite{AKM12, HerMat15i, HerMat15ii} as well as techniques for
joining together peicewise-defined Lyapunov functions in an
analytically simple way from  \cite{HerMat15i, HerMat15ii}.

In Section~\ref{sec:reduced}, we state the main results of the paper
which are derived from the existence of an appropriate Lyapunov
function. Section~\ref{LyapOverview} gives an overview of the
construction of the Lyapunov function as well as some heuristic
descriptions of its origin. More specifically in Section~\ref{sec:Numericalexplorations},
we present some numerical experiments which show that our  Lyapunov
function is in some sense correct.  In Section~\ref{secBasicIdea}, we give a digestible overview of the
basic ideas used in the construction while in Section~\ref{sec:Hypocoercivity}, we give some indications of the
relation between the ideas discussed in Section~\ref{secBasicIdea} and the ideas of
hypocoercivity.  In Section~\ref{ApproxDyn}, we introduce the approximate dynamics which makes the analysis outlined in
Section~\ref{secBasicIdea} feasible. The actual Lyapunov function is defined in Section~\ref{sec:PoisonEquations} in terms of solutions of Poisson equations associated to the approximate dynamics introduced in Section~\ref{ApproxDyn}.  In Section~\ref{sec:consequences}, we give some consequences of the Lyapunov structure we have proven. In Section~\ref{sec:lyap_proof} and
the Appendix, we give the missing details from the proof that the
candidate function constructed is in fact a proper Lyapunov function.  We conclude in Section~\ref{sec:conclusion} by briefly discussing the
challenges of extending our results to larger systems and the case of a
harmonically growing potential which is not covered by our results.

\section{A Model Problem}
\label{ModelProblem}


Consider the two-particle Hamiltonian system
$(\mathbf{Q},\mathbf{P})= \left((Q_1,Q_2),(P_1,P_2)\right)$ with
Hamiltonian 
\begin{align*}
  \Hfull(\mathbf{Q},\mathbf{P}) = \frac{P_1^2}{2} +  \frac{P_2^2}{2} +  U(Q_1-Q_2)\end{align*}
and interaction potential
\begin{align}\label{Udef1}
  U(Q) = \sum_{j=1}^l a_j |Q|^{\alpha_j} >0, 
\end{align}
where $a_j \in \R$ with $a_1,a_{l}>0$, and $\alpha_1 > \cdots >
\alpha_l$. We assume that $\alpha_1 > 2$ and $\alpha_l <0$ (otherwise
no singularity exists).  The dynamics of this system is given by 
\begin{equation*}
  \begin{aligned}
       \dot Q_i=\frac{\partial \Hfull}{\partial P_i} \qquad
       \dot P_i &=-   \frac{\partial \Hfull}{\partial Q_i} \qquad
       \text{for $i=1,2$.}
  \end{aligned}
\end{equation*}

If we force the system with a noise whose magnitude is scaled to balance 
dissipation so as to place the system at temperature $T$, then we arrive at the system of coupled SDEs
\begin{equation} \label{SDE4D}
  \begin{aligned}
  dq_i&=p_i\, dt\qquad \text{for $i=1,2$} \\
  dp_1&= - U'(q_1-q_2)\,dt - \gamma p_1\,dt+\sigma dW_1(t)\\
  dp_2&= \;U'(q_1-q_2)\,dt - \gamma p_2\,dt +\sigma dW_2(t)
\end{aligned}
\end{equation}
where the friction $\gamma>0$ and $\sigma^2= 2 \gamma T$. 
Define 
\begin{align*}
 \mathbb{S}\eqdef\Big\{  (p_1,q_1,p_2,q_2) : q_1 \neq q_2 \Big\}.
\end{align*}
We will prove in Corollary~\ref{S-existance} that, if the initial
conditions are in $\mathbb{S}$, then with probability one there exists
a unique strong solution to \eqref{SDE4D} which is global in time and
stays in $\mathbb{S}$.

We define the Markov semigroup by
$(\mathcal{P}_t\phi)(\mathbf{p},\mathbf{q}) \eqdef
\E_{(\mathbf{p},\mathbf{q}) } \phi(\mathbf{p}_t,\mathbf{q}_t)$ where
$\E_{(\mathbf{p},\mathbf{q}) }$ is the expected value starting from
$(\mathbf{p},\mathbf{q})$ . This semigroup has a generator $\Lfull$ given
by
\begin{equation*}
  \Lfull \; \eqdef \sum_{i=1,2} \frac{\partial \Hfull}{\partial p_i}
    \frac{\partial}{\partial q_i} - \frac{\partial \Hfull}{\partial
      q_i} \frac{\partial}{\partial p_i} -\gamma p_i
    \frac{\partial}{\partial p_i} +\gamma T
    \frac{\partial^2}{\partial p_i^2}\,.
\end{equation*}
Additionally $\mathcal{P}_t$ induces a dual action on $\sigma$-finite
measures $\mu$ by acting on the left: $\mu \mathcal{P}_t$.  A measure $\mu_0$ is a
stationary measure of $\mathcal{P}_t$ if $\mu_0\mathcal{P}_t =\mu_0$.
In our setting, this is equivalent to asking that $\Lfull^* \rho_0 =0$
where $\mu_0(d\mathbf{p},d\mathbf{q}) = \rho_0(\mathbf{p},\mathbf{q})d\mathbf{p}\,d\mathbf{q}$.

It is a simple calculation to  see that if
\begin{align*}
\rho_0(\mathbf{p},\mathbf{q})\, \eqdef \,Ce^{-H_0(\mathbf{p},\mathbf{q})/T } 
\end{align*} 
for any $C$,
then  $\Lfull^*
\rho_0(\mathbf{p},\mathbf{q})=0$. Hence with this choice of $\rho_0$,
$\mu_0$ as defined above is a stationary measure.  However this
measure is not normalizable to make a probability measure since it is
only $\sigma$-finite. This stems from the fact that the Hamiltonian
is translationally invariant in $\mathbf{q}$. To rectify his  problem we will move to ``center of mass''
coordinates. 

\subsection{Reduction to Center of Mass Coordinates}
Let $\tilde q=\frac12(q_1-q_2)$, $\tilde
p=\frac12(p_1-p_2)$, 
$\bar{q}=\frac12(q_1+q_2)$, $\bar{p}=\frac12(p_1+p_2)$, $W=\frac12
(W_1 - W_2)$ and $B=\frac12 (W_1 + W_2)$. Then
\begin{equation}\label{eqSystem}
\begin{aligned}
  d\bar{q}_t&=\bar{p}_t \, dt\\
  d\bar{p}_t&= -\gamma  \bar{p}_t \, dt+\sigma dB_t\\
  d\tilde q_t&= \tilde p_t\, dt\\
  d\tilde p_t&= - U'(2 \tilde q_t)\, dt- \gamma \tilde p_t\, dt +\sigma dW_t   
\end{aligned}
\end{equation}
In these new coordinates, the system is described by variables
$(\bar{q},\bar{p})$ tracking the position and momentum of the center of mass,
and variables $(\tilde q,\tilde q)$ tracking the relative position and momentum
of the particles within the center of mass frame. This change of
coordinates simplifies our problem to two uncoupled Hamiltonian
sub-problems. The center of mass $(\bar{q},\bar{p})$, has Hamiltonian
\begin{align*}
 \Hharm(\bar{q},\bar{p})\, \eqdef \, \frac{\bar{p}^2}{2}
\end{align*}
which is the Hamiltonian of a free 1D particle, with corresponding
invariant measure given by a Gaussian (for momentum $\bar{p}$) times
1D Lebesgue measure (for position $\bar{q}$). Note that $\bar{p}$ follows an
Ornstein-Uhlenbeck process and hence converges exponentially quickly
to its (Gaussian) stationary measure. The position $\bar{q}$ will
diffuse through space like 1D Brownian motion and hence converges to
Lebesgue measure.

The remaining two variables $(\tilde q,\tilde p)$ are also a
Hamiltonian system with Hamiltonian 
\begin{align}\label{Hreduced}
  H(\tilde q,\tilde p)\,\eqdef\,\frac{\tilde p^2}{2} + U(2 \tilde q)\,.
\end{align}
which is a single particle interacting with a potential $U$ that
is attractive towards the origin at large distances, and
  repulsive at short distance.  So $(\tilde q,\tilde p)$ will have an invariant
{\it probability} measure. However convergence of this system is more
subtle; it possesses two difficulties stemming from the structure of
the potential. First, since $U(Q)$ is singular at points, a strictly
positive density does not exist everywhere in space. Second, there is
no immediate candidate for a Lyapunov function. Overcoming this second
obstacle will prove more difficult and will occupy the bulk of this
paper.

\section{Reduced system: main results}\label{sec:reduced}
We now turn to the study of the two--dimensional Hamiltonian system
described by \eqref{Hreduced}. In this section, we also state the principal
results on this reduced system.

Consider the two-dimensional deterministic Hamiltonian system with
Hamiltonian
\begin{align*}
H(Q,P)\eqdef\frac{P^2}2  + U(Q)
\end{align*}
and hence dynamics 
\begin{equation*}
 \begin{aligned}
   \dot Q_t=\frac{\partial H}{\partial P}(Q_t,P_t) =P_t\qquad\text{and}\qquad\dot P_t
   &=- \frac{\partial H}{\partial Q}(Q_t,P_t)= - U'( Q_t)\,.
 \end{aligned}
\end{equation*}
This system has only closed orbits, which lie completely in the
  upper half plane
denoted by $\H=\{ (Q,P) \in \R^2 : Q > 0\}$
provided the initial points lie in $\H$. To see this observe that when
$|(Q,P)| \rightarrow \infty$, $H(Q,P)$ is well approximated by
$\frac{1}{2}P^2+a_1Q^{\alpha_1} + a_KQ^{\alpha_l}$ which clearly has
level sets that are closed, homotopically a circle, and lie
completely in the upper half plane. (See Figure~\ref{Fig:levelconverge}).
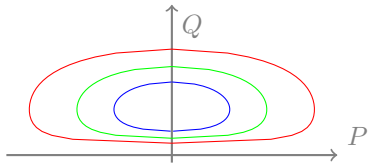
\begin{figure}
\centering
  \begin{tikzpicture}[scale=1,domain=-.1:2,rotate=90]
      \draw[->,thick,color=gray] (-.1,0) -- (2,0) node[below right] {$Q$};
      \draw[<-,thick,color=gray] (0,-2.2) node[above right] {$P$} -- (0,2.2);
     \draw[color=blue,thin]	plot[id=bottomBlue,domain=0.3179:0.97245]
     function{-sqrt(abs(1-x**4-.1/x**2))};
 \draw[color=blue,thin]	plot[id=topBlue,domain=0.3179:0.97245]
   function{sqrt(abs(1-x**4-.1/x**2))};  
   \draw[color=green,thin]	plot[id=topGreen,domain=0.2237470347:1.178353708]
  function{sqrt(abs(2-x**4-.1/x**2))};  
    \draw[color=green,thin]	plot[id=bottomGreen,domain=0.2237470347:1.178353708]
   function{-sqrt(abs(2-x**4-.1/x**2))};
   \draw[color=red,thin]	plot[id=topRed,domain=0.1581262410:1.409744948]
  function{sqrt(abs(4-x**4-.1/x**2))};  
    \draw[color=red,thin]	plot[id=bottomRed,domain=0.1581262410:1.409744948]
   function{-sqrt(abs(4-x**4-.1/x**2))};
   \end{tikzpicture}
\caption{ Level sets of $H(Q,P)=\eta$ for $\eta$ equals 1 (in blue), 2
  (in green), and 4 (in red)
where $H(Q,P)=\frac{1}{2} P^2+ Q^4 +\frac{1}{10} Q^{-2}$. }
\label{Fig:levelconverge}
\end{figure}

Addition of balanced noise and dissipation yields the associated
stochastic system of interest. Namely, for positive temperature $T$,
friction $\gamma$ and noise standard deviation
$\sigma=\sqrt{2\gamma T}$, we have
\begin{equation}
 \label{sde1}
 \begin{aligned}
   dq_t&=  p_t\,dt \\
   dp_t&= -U'(q_t) \,dt - \gamma p_t \,dt +
   \sigma \,dW_t \,.
 \end{aligned}
\end{equation}
This Markov process has generator
\begin{equation*}
  \mathcal{L} \; =   \frac{\partial H}{\partial p}  \frac{\partial\
       }{\partial q} -   \frac{\partial H}{\partial q}   \frac{\partial\
       }{\partial p}  -\gamma  p \frac{\partial}{\partial p} +  \gamma T  \frac{\partial^2}{\partial p^2}
\end{equation*}
and as in the previous section a straightforward calculation shows
that $\mu_*(dp\times dq)=\rho_*(q,p)dpdq$ is a
stationary measure with
\begin{align}
\label{eqn:gibbs}
  \rho_*(q,p) = Ce^{-H(q,p)/T},
\end{align}
since $\mathcal{L}^*  \rho_*=0$.  Unlike the stationary
measure of the unreduced system, this measure can be normalized and
made into a probability measure  for an appropriate choice of $C$ (since $H$ is no longer
translationally invariant).

In fact $\rho_*$ is the unique stationary measure of the system.
To see this first observe that \eqref{sde1} is
hypoelliptic and hence any weak solution to $\mathcal{L}^*\mu=0$
must locally have a smooth density with respect to Lebesgue measure.
Since $ \mu_*$ has an everywhere positive density with respect
to Lebesgue measure it must therefore be the only stationary measure,
since any stationary measure can be decomposed into its ergodic
components all of which must have disjoint support. Uniqueness of the
stationary measure is also a by-product of the exponential convergence
given in Theorem~\ref{MainThm} which is our main interest here.

To state this convergence result we need a distance between
probability measures appropriate for our setting. To this end, for any
$c \geq 0$ we define for $\phi:\H \rightarrow \R$ the weighted
supremum-norm
\begin{align*}
  \|\phi\|_{c} \eqdef \sup_{(q,p) \in \H}
  |\phi(q,p)|e^{-c H(q,p)}
\end{align*}
and the weighted total-variation norm on signed measures
$\nu$ with the property that $\nu(\H)=0$ by
\begin{align*}
  \|\nu\|_{c} \eqdef \sup_{\phi : \|\phi\|_{c} \leq 1} \int_{\H} \phi\, d\nu\,.
\end{align*}
When $c=0$  this is just the standard total-variation norm.
We define $\mathcal{M}_{c}(\H)$ to be the set of probability
measures $\mu$ on $\H$ with $\int_{\H} \exp(c H)
d\mu < \infty$. Then we have the following convergence result.
\begin{theorem}\label{MainThm}  For any $c\in (0, 1/T)$, there exist positive constants $C$ and
  $D$ such that for any two probability measures $\mu_1, \mu_2 \in
  \mathcal{M}_{c}(\H)$
  \begin{align*}
    \|\mu_1\mathcal{P}_t -\mu_2 \mathcal{P}_t \|_{c} \leq C e^{-D t} \|\mu_1-\mu_2\|_{c}
  \end{align*}
  for all $t\geq 0$.  In particular the system has a unique invariant measure, which
  necessarily coincides with $\mu_*$ defined above, and to which
    the distribution of ($q_t,p_t)$ converges exponentially fast.
\end{theorem}
%
%
Our proof of Theorem~\ref{MainThm}
will follow the now standard approach of establishing the existence of
an appropriate ``small set'' and a Lyapunov function
\cite{MeynTweedie93}. Similar to \cite{MattinglyStuartHigham02}, we
will use a control argument coupled with hypoellipticity to establish
the existence of a small set.  While this is rather standard, the
technique used to prove the existence of a Lyapunov function is less
standard and one of the central contributions of this paper.


\section{The Lyapunov function: Overview}
\label{LyapOverview}
\subsection{Heuristics and motivating discussion}
\label{Heuristics}
We wish to control motion out to infinity ($|(q,p)| \rightarrow \infty$) 
as well as in the neighborhood of the singularity ($q \rightarrow
0^+$). A standard route to obtaining such control is to find a
Lyapunov function
$V:\H \rightarrow (0,\infty)$ so that 
\begin{align}\label{LyapGen}
  dV(q_t,p_t) \leq -c V(q_t,p_t) dt + Cdt +dM_t
\end{align}
for some martingale $M_t$ and positive constants $c,C$ and such that
$H \leq  C_0 V$ for some positive $C_0$.  In particular, the fact that
$V \rightarrow \infty$ as $q \rightarrow 0^+$ allows us to control the
time spent near $q=0$.
 
The first reasonable choice for a Lyapunov function might be to try the Hamiltonian
$H(q,p)$ itself. Using It\^o's formula, we see that 
\begin{align}\label{Heq}
 d H(q_t,p_t) = -\gamma p_t^2dt + \frac{\sigma^2}2 dt + \sigma p_t dW_t\,.
\end{align}
However the function $(q,p) \mapsto p^2$ is not bounded below by
$(q,p) \mapsto H(q,p)$ since the two functions are not comparable.
This prevents us from obtaining the desired bound.
If $U(q)$ only has positive powers of $q$ that are greater or equal
to two, this deficiency can be partially overcome by considering
$V(q,p)=H(q,p) + \gamma_0 p q$. Then by picking $\gamma_0$ small enough,
we can ensure that $\frac1c H\leq V \leq c H $ as $p^2 + q^2
\rightarrow \infty$ and that $\mathcal{L} V$ is bounded from above by
a constant times $-V+ C$ for some $C>0$. Hence $V$ is comparable to
$H$ but satisfies the desired Lyapunov-function inequality
\eqref{LyapGen}. See \cite{MattinglyStuartHigham02} for more on using
this trick in this context.

Unfortunately this simple trick does not work in the presence
of a singular repulsive term, as it does not yield the required bound for geometric ergodicity when $q$ approaches 0. This is necessary since the potential, and hence
the transition density, behaves poorly near this point and uniform
estimates are not easy (if even possible) to obtain.  It is therefore reasonable to
ask if there is a different choice other than $pq$ that will work yet is
inspired by this example.  Eventually, we will find an appropriate
function $\Psi$ so that $V=H + \Psi$ works; to do so we will leverage
a better understanding the dynamics at large energies. Moreover, this will allow us to learn a different way to 
understand the $pq$ correction than
via the theory of hypocoercivity which it motivated.  In Section~\ref{sec:Hypocoercivity},
we will return to this example which is
connected to the  theory of hypocoercivity, which it partially
inspired, and see how it fits into the approach we have developed.

With this example and its limitations in 
mind, we return to \eqref{Heq} and take a closer look at the
dynamics. Looking at the right hand side, it is true that $p^2$ is not
comparable to $H(q,p)$ at every given point $(q,p)$ in phase
space. Yet if we really believe that the system settles down into equilibrium exponentially fast, the
$-p^2$ term must lead to some ``dissipation''  of energy when the energy is
large.

To see how dissipation arises, it is sufficient to analyze the stochastic dynamics at large energies, which is a regime in which we know something about the dynamics.  To leading order in $H$ it will follow the deterministic dynamics with stochastic fluctuations of lower order. At high energy, the highest order part of the potential $U$ dominates.

For discussion purposes, we will assume for the moment that the potential $U: \H \rightarrow (0, \infty)$ has the simplified form
\begin{align}\label{Umodel}
U(Q)=a Q^{\alpha} + b Q^{-\beta}
\end{align}
for some $a, b >0$ and $ \alpha, \beta  >0 $ with $\alpha>2$.  Later in this section, we will return to the problem when $U(Q)$ has the more general form~\eqref{Udef1}.  It will be convenient to introduce the following family of potentials indexed by a parameter
$\epsilon \in[0,1]$
\begin{align*}
  U_\epsilon(Q)=aQ^{\alpha} + b Q^{-\beta} \epsilon^{1+\frac{\beta}{\alpha} }. 
\end{align*}
Setting $\epsilon=1$ yields the original potential which we will
continue to denote by $U$ without any subscript.  The advantage provided by
considering this family of potentials is that $U_\epsilon(Q)$ has the
following homogeneous scaling property for $h>0$
\begin{align}\label{UScaling}
  U_\epsilon(h^{\frac{1}{\alpha}} Q)=h U_{\frac{\epsilon}{h}}(Q),
\end{align}
and this scaling property will lead to all of the scaling properties mentioned
subsequently.

The orbits of the deterministic trajectories are given by the
solution set of $H_\epsilon(Q,P) = \frac12
P^2+U_\epsilon(Q)= \eta$ for a given energy level $\eta >0$.  This
locus is topologically equivalent to a circle and hence setting
\begin{align}\label{solutionDef}
  \Qcurve_\epsilon(Q,\eta)= \sqrt{2(\eta-U_\epsilon(Q))},
\end{align}
the orbit is given by the set $\{ (Q,\Qcurve_\epsilon(Q, \eta)),
(Q,-\Qcurve_\epsilon(Q, \eta)) : Q
\in [Q^\epsilon_-(\eta),Q^\epsilon_+(\eta)] \}$ where
$Q^\epsilon_+(\eta)$ and $Q^\epsilon_-(\eta)$ are respectively the
largest and smallest positive roots of $\eta-U_\epsilon(Q)=0$. Notice
that model potential we are currently considering always has exactly two
solutions to $\eta-U_\epsilon(Q)=0$.

We will see that  the period of the orbit goes to zero as the energy
goes to infinity. Hence at high energy the system will make many
orbits in an instant of time and the average of $-P^2$
around the deterministic orbits will give a good idea of
the dissipation asymptotically as the energy becomes large. We see
that averaging $P^2$ around this deterministic trajectory gives by symmetry
\begin{align*}
  \langle P^2 \rangle_\epsilon(\eta) =
  2\int_{Q^\epsilon_-(\eta)}^{Q^\epsilon_+(\eta)}   \Qcurve_\epsilon(Q,
  \eta) dQ\, ;
\end{align*}
and similarly that the period $\tau_\epsilon(\eta)$ of this orbit can be expressed as
\begin{align*}
  \tau_\epsilon(\eta)=2\int_{Q^\epsilon_-(\eta)}^{Q^\epsilon_+(\eta)}
  \frac1{\Qcurve_\epsilon(Q, \eta)}dQ .  
\end{align*}

To make the idea of ``large energy'' more precise we consider the
rescaling of phase space defined by the mapping $(Q,P) \mapsto (h^{\frac12}
P,h^{\frac{1}{\alpha}} Q)$ for a scale factor $h >0$. Under
this map, the associated energy will essentially scale by a factor
$h$ for large $h$. However this is not \textit{exactly} correct
since the other terms in the potential do not scale in the same
fashion. However, in light of \eqref{UScaling},  by changing the value of $\epsilon$ we can relate a
scaled Hamiltonian exactly with an unscaled Hamiltonian having
$\epsilon=h^{-1}$; that is, since $H_\epsilon(Q,P) = \frac12 P^2 +
U_\epsilon(Q)$, we see that $H_\epsilon(h^{\frac1{\alpha}} Q,h^{\frac12} P)= h H_{\frac{\epsilon}{h}}(Q,P)$. 
In other words, the scaled system behaves exactly like the
  unscaled system at a higher energy. If we define the average value
  of $P^2$ about an orbit as 
\begin{equation} \label{eq:defn-acal}
	\mathcal{A}_\epsilon(P^2)( \eta)\eqdef\frac{\langle
	P^2\rangle_\epsilon(\eta)}{\tau_\epsilon(\eta)}
\end{equation} 
then we also see that $\mathcal{A}_\epsilon(P^2)( h \eta)=h
\mathcal{A}_{\frac{\epsilon}{h}}(P^2)(\eta)$.


Summarizing, the average of $P^2$ around the deterministic orbit with
 energy $h \eta$ and $\epsilon=1$ is the same as $h$ times the average
 of $P^2$ around the deterministic orbit with energy $\eta$ and
 $\epsilon=h^{-1}$  for the simplified potential considered in this
 section. We will see later that this will hold for sufficiently large
 energy for the more general potential~\eqref{Udef1} as well.
If we define, 
\begin{align}\label{eq:rate}
  \rate(\eta)\eqdef\mathcal{A}_{\frac1{\eta}}(P^2)(1)
\end{align}
 then $\mathcal{A}_1(P^2)(\eta)=
 \eta \rate(\eta) $. Furthermore, observe that as $\epsilon \rightarrow 0$, the level sets under potential
   $U_\epsilon(Q)$ converge (Figure~\ref{Fig:LambdaLevelSets}), and $\mathcal{A}_\epsilon(P^2)( 1)$ converges
 to a positive constant $\rate_*$ as $\epsilon \rightarrow0$. 
As we will see later
\begin{align}\label{rateStar}
  \rate_*=\frac{\int_0^{\widetilde{Q}} \big(1 -a Q^{\alpha}
    \big)^{\frac12}dQ}{\int_0^{\widetilde{Q}} \big(1 - a Q^{\alpha} 
    \big)^{-\frac12}dQ}=\frac{2 \alpha}{\alpha+2}
\end{align}
where $\widetilde{Q}=a^{-\frac{1}{\alpha}}$.  
Notice that $\rate_*$ is
independent of the value of $a$ and since $\alpha > 2$, observe
that $\rate_* \in (1,2)$.

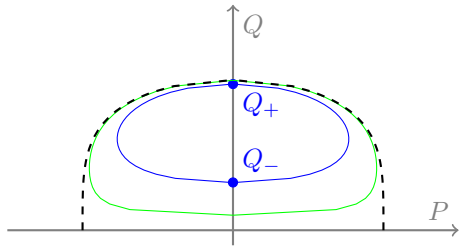
\begin{figure}
\centering
\begin{tikzpicture}[scale=2,domain=-.1:1.5,,rotate=90]
    \draw[->,thick,color=gray] (-.1,0) -- (1.5,0) node[below right] {$Q$};
    \draw[<-,thick,color=gray] (0,-1.5) node[above left] {$P$} -- (0,1.5);
   \draw[color=blue,thin]	plot[id=bottomBlueII,domain=0.3179:0.97245]
   function{-sqrt(abs(1-x**4-.1/x**2))};
\draw[color=blue,thin]	plot[id=topBlueII,domain=0.3179:0.97245]
 function{sqrt(abs(1-x**4-.1/x**2))};  
   \draw[color=blue,fill]  (0.3179 ,0) circle (.2ex);
    \node[above right,color=blue] at  (0.3179 ,0) {$Q_-$};  
    \draw[color=blue,fill]  (0.97245 ,0) circle (.2ex);
    \node[below right,color=blue] at  (0. 97245,0) {$Q_+$};
  \draw[color=green,thin]	plot[id=topGreenII,domain=0.1000050014:.9974778151]
 function{sqrt(abs(1-x**4-.01/x**2))};  
   \draw[color=green,thin]	plot[id=bottomGreenII,domain=0.1000050014:.9974778151]
  function{-sqrt(abs(1-x**4-.01/x**2))};
  \draw[color=black,thick,dashed]	plot[id=topBlackII,domain=0:1]
 function{sqrt(abs(1-x**4))};  
   \draw[color=black,thick,dashed] plot[id=bottomBlackII,domain=0:1]
  function{-sqrt(abs(1-x**4))};
  \end{tikzpicture}
\caption{ Level sets of $H_\epsilon(Q,P)=1$ for $\epsilon$ equals 1 (blue) and
 1/2.15 (green) where $H_\epsilon(Q,P)=\frac{1}{2} P^2+ Q^4 +
 \frac1{10}\epsilon^{\frac{3}{2}} Q^{-2}$  The dashed line is the level set of
 $\frac{1}{2} P^2+ Q^4 =1$ with $P \geq 0$ to which the level sets of
 $H_\epsilon(Q,P)=1$  converge as $\epsilon \rightarrow 0$.
}
\label{Fig:LambdaLevelSets}
\end{figure}

Now since at high energy (i.e. $\eta \gg 1$), $\mathcal{A}_1(P^2)( \eta )= \eta
\rate(\eta)\approx \eta
\rate_*$, it is reasonable to 
approximate \eqref{Heq} by
\begin{align}\label{HeqApprox}
  d H(t) \approx -\gamma\rate_*  H(t) dt + \zeta \, dt + \sigma
  \sqrt{\rate_* H(t)} dW(t)
\end{align}
when $H(t) \gg 1$ where $\zeta >0$ is constant.  Note that $\zeta $ is negligible for $H(t) \gg 1$.  The martingale in \eqref{HeqApprox} was chosen so
that its quadratic variation would be the time average of the quadratic variation of the martingale in \eqref{Heq}.
In making this approximation, we are \textit{not} claiming that there
is averaging in the traditional asymptotic sense. Namely,
there is a small parameter going to zero that causes the
\textit{whole} system
to speed up and hence the instantaneous effect on the system is
increasing in the limit of that averaged parameter.
Rather, at high energy the system acts (after rescaling) increasingly
like a system with order one energy and a rescaled parameter
$\epsilon$. The rescaling also leads to a rescaling  of time so that an
order one time in the rescaled system represents an increasingly short
time in the original system. Hence in a short interval of time at high
energy, one sees the effect of many rotations of the system, making
the averaged quantities just calculated a good approximation. 

In spirit this approach is initially not unlike one used to show
stability of queuing systems and stochastic algorithms
\cite{DupuisWilliams1994,HuangKontoyiannisMeyn2001,Meyn2008}. There a
discrete time (and possibly discrete space) stochastic system is shown
to converge after rescaling to a deterministic ODE which can easily be
shown to be stable. Here we also rescale but do so primarily to
introduce a small parameter (one over the energy) and then use
averaging the study this limiting ODE system with a small
parameter. 

\begin{figure}
\includegraphics[width=5 in]{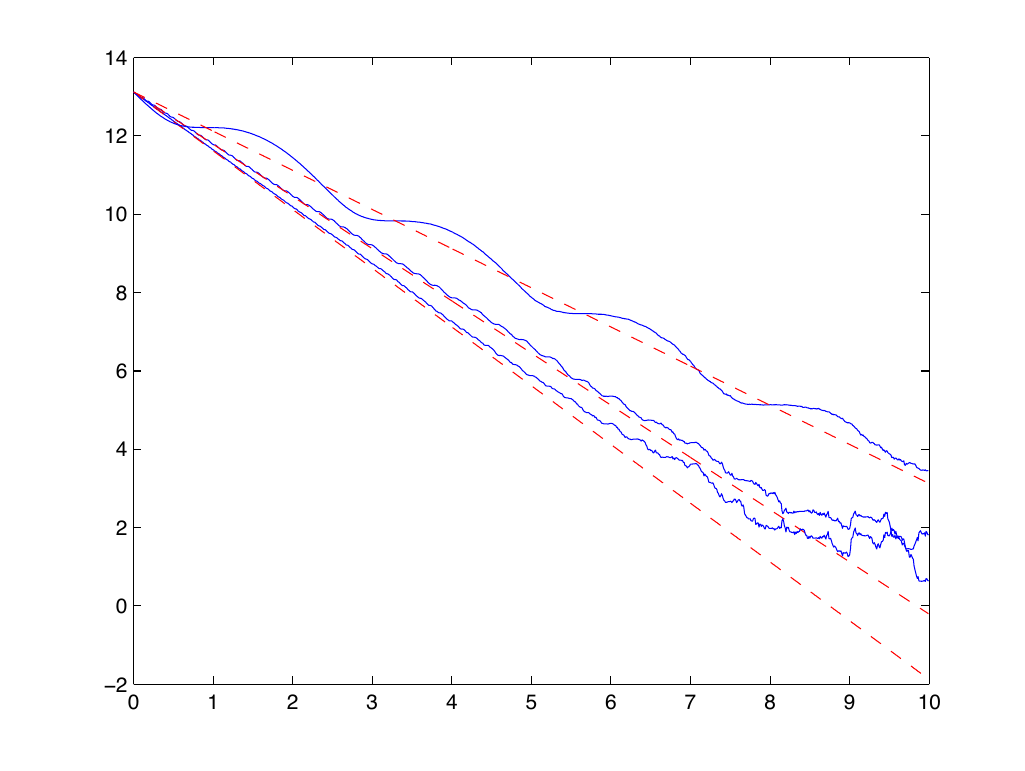}
\caption{The first three plots are semi-log plots of energy versus
 time for the dynamics using the potential in \eqref{Umodel} with
 $\alpha$ equal to 2 (upper most curve), 4 (middle curve), and 6
 (lower most curve).The solid lines are
 numerical simulations and the dashed lines are the theoretical
 prediction made by \eqref{HeqApprox}. }
\label{fig:numerics}
\end{figure}
Before  making this intuition more formal in Section~\ref{ApproxDyn},
we will present some numerical experiments which show that the above
calculations capture the ``truth'' of what is going on. We will see the
give the observed rate of energy dissipation at high energies.
\subsection{Numerical explorations}
\label{sec:Numericalexplorations}

The plots in Figure~\ref{fig:numerics} compare
the trajectory of the energy predicted by \eqref{HeqApprox} and the
energy trajectory obtained from a numerical simulation of \eqref{sde1}
when both were started from the same initial high energy level. The
model potential given in \eqref{Umodel} was used with $\alpha \in \{2,
4,6\}$ and $\beta =12$.   Similar comparisons with  $\beta$ equal
to $2$ and $4$ were also made with nearly identical plots confirming
essentially no dependence on $\alpha$ as predicted by our asymptotic
theory. 

Our theory only applies to the two cases $\alpha \in \{4,6\}$ since
the theory requires $\alpha>2$. In
these cases the agreement with the theory, shown with the dashed line,
is quite good. One can see a small scale wiggle in the numerical
curves. This is the effect of the periodic orbit. As the scaling
theory predicts, the effect decreases as the energy increases since
the scaling shows that period and the size of the fluctuations go to
zero as the energy increases. When $\alpha=2$ our theory does not
apply. Nonetheless, the trend given by dotted line is
followed. However one sees that period and amplitude of the
fluctuation is not going to zero which is also consistent with the
scaling arguments predictions. The possibility of extending our theory
to this boundary case is discussed in Section~\ref{sec:conclusion}.

\subsection{Definition of the Lyapunov function}
Informed by the preceding discussion, we return to the idea of
constructing a Lyapunov function $V$ of the form $V=H+\Psi$, where
  $\Psi$ is introduced to handle the singularity in $H$.  The end
  result of this section, in particular, will be the definition of the
  corrector $\Psi$.  First, however, we will take time to both
  motivate and explain how we arrived at this definition.

  As discussed in Section~\ref{Heuristics}, at high energy the system
  moves essentially around the deterministic orbit defined by the
  Hamiltonian flow. The average dissipative effect of each of these orbits is
  given by the average of the right hand side of \eqref{Heq} around
  one orbit. In the language of \eqref{eq:defn-acal}, this is
  $-\mathcal{A}_1 (P^2)(h)+ \frac{\sigma^2}{2}$ if the energy
  equals $h$. To replace the $-p^2$ from \eqref{Heq} with
  $-\mathcal{A}_1 (P^2)(h)$, the theory of
homogenization and averaging suggest the use of the ``corrector'' $\Psi$
defined by Poisson equation
  \begin{align*}
    \mathcal{H} \Psi(q,p) = \gamma(p^2 - \mathcal{A}_1 (P^2)(H(p,q)))\,.
  \end{align*}
where  $\mathcal{H}$ is the Liouville operator defined below. This can
also be thought of as an ``integration by parts'' adapted to
deterministic Hamiltonian dynamics in this setting, in the sense that 
\begin{align*}
  \int_0^t \gamma P_s^2\, ds = \Psi(P_t,Q_t) - \Psi(Q_0,P_0) + t \gamma \mathcal{A}_1 (P^2)(H(Q_0,P_0))\,.
\end{align*}
The first two terms on the righthand side of the equation above are boundary terms which control the
fluctuations from the mean value.

This is the argument used in
\cite{HairerMattingly:2009}, where a succession of Poisson equations
was employed to produce a sequence of correctors to reduce the fluctuations
in various terms, achieving a function which was pointwise
dissipative/coercive.  In many ways the situation here is simpler
than in \cite{HairerMattingly:2009} and the presentation
clearer. However, we will see that a number of needed estimates proved
elusive in this simple program as presented above. We will need to
modify the above arguments by combining them with ideas the works~\cite{AKM12, HerMat15i, HerMat15ii}.

 \subsubsection{The basic idea}
 \label{secBasicIdea}
We begin by introducing the Liouville operator $\mathcal{H}$ associated with the deterministic dynamics given by 
\begin{align}\label{eq:genH}
\mathcal{H} \eqdef P \partial_Q - U'(Q) \partial_P .
\end{align}   
Recalling that the full stochastic dynamics at large energies is approximately determined by the dynamics along $\mathcal{H}$, ideally we would like to pick the corrector $\Psi$ so that it satisfies the following two properties:
\begin{itemize}
\item[(I)] $\Psi(Q,P)\in C^2(\H: \R)$ and $\Psi$ satisfies the following PDE on $\H$
\begin{align}\label{eq:ideaPoison}
(\mathcal{H} \Psi)(Q,P) = \gamma(P^2- \mathcal{A}_1(P^2)(Q,P))
\end{align} 
where $\mathcal{A}_\epsilon(P^2)( \eta)$ is the averaging operator defined by \eqref{eq:defn-acal} discussed in Section~\ref{Heuristics} and we have introduced the slight abuse of notation
\begin{align*}
  \mathcal{A}_\epsilon(P^2)(Q,P)\eqdef \mathcal{A}_\epsilon(P^2)(H_\epsilon(Q,P))\,.
\end{align*}
\item[(II)]  $\Psi(Q,P)$ is ``asymptotically dominated" by $H(Q,P)$ as $H(Q,P)\rightarrow \infty$, i.e., $\Psi$ satisfies
\begin{align*}
\Psi(Q,P) = o(H(Q,P)) \,\, \text{ as } \, \, H(Q,P) \rightarrow \infty.  
\end{align*} 
\end{itemize}
In a moment, we will remark as to why we need to slightly weaken
property (I) here, but for now let us assume that such a $\Psi$
satisfying (I) and (II) exists, as the essential structure of the
argument that follows will still be employed.

Recall that that the generator $\mathcal{L}$ of the process defined by
\eqref{sde1} can be written as
\begin{align*}
\mathcal{L}= \mathcal{H}- \gamma p \, \partial_p + \frac{\sigma^2}{2}\, \partial_p^2 .
\end{align*}
As mentioned above, we will choose the Lyapunov function $V$ to be 
$V=H+\Psi$.  Since
$\Psi$ satisfies the PDE in equation \eqref{eq:ideaPoison} of property
(I), $\mathcal{A}_1(P^2)( \eta) =
\rate(\eta)  \eta$ and $\mathcal{H} H =0$, we have that
\begin{align}\label{Veq}
  dV(q_t,p_t) = (\mathcal{L}V)(q_t,p_t)dt + dM_t,
\end{align}
where $M_t$ is a local martingale and 
\begin{multline}\label{LV}
  (\mathcal{L}V)(q,p) =- {\gamma}(\rate\circ H)(q,p)H(q,p) +
  \frac{\sigma^2}{2} - \gamma p \frac{\partial \Psi}{\partial
    p}(q,p) + \frac{\sigma^2}{2} \frac{\partial^2\Psi}{\partial
    p^2}(q,p) \,.
\end{multline}
The first two terms of the right-hand side of \eqref{LV} essentially
coincide with \eqref{HeqApprox}; therefore, to realize our goal we
would need to show that the
remaining terms on the right-hand side are negligible at large
energies. 

To see intuitively why we expect these terms to be negligible at large energies, set $\beta=0$ in the 
potential $U(q)$ for simplicity and note that the operator $\mathcal{H}$ scales homogeneously of 
degree $\frac12-\frac1\alpha$ under the transformation $(P,Q)\mapsto
(h^{\frac12}P,h^{\frac1\alpha}Q)$. Also, notice that the Hamiltonian $H$
scales homogeneously of degree $1$ under this transformation. Since the right hand side of
\eqref{eq:ideaPoison} scales homogeneously of degree $1$ under the same
transformation, we expect the corrector $\Psi$ to scale like $h^{\frac12 +
  \frac1\alpha}$. Since we assumed that $\alpha >2$, we see that
(when $\beta=0$) $\Psi$ is dominated by $H$ at large energies just from
this argument. Similarly, we expect $ P\partial_P
  \Psi$ and  $\partial^2_P\Psi$ to scale respectively like $h^{\frac12 +
  \frac1\alpha}$ and $h^{\frac1\alpha-\frac12}$ under the same scaling, and
hence are negligible as previously claimed.

When $\beta >0$, however, the situation is more complicated. A nice $C^2$
solution to \eqref{eq:ideaPoison} can still be found, yet
determining its behavior at large energies is more delicate. For large energies where
$q^\alpha$ dominates, the above analysis should still
hold.  For large energies where $q^{-\beta}$ dominates in $U(q)$, one can
change the parameter $\epsilon$ in $U_\epsilon(q)$ from
\eqref{UScaling} to perform a similar scaling analysis for solutions
of \eqref{eq:ideaPoison} with $U$ replaced by $U_\epsilon$. 
More precisely, if one defines $\mathcal{H}_\epsilon$ by \eqref{eq:genH}
with $U'(q)$ replaced by   $U_\epsilon'(q)$, then under the scaling
transformation $(P,Q) \mapsto (h^{\frac12} P, h^{\frac1\alpha} Q)$ we
have that $\mathcal{H}_\epsilon$ transfroms to $h^{\frac12
  -\frac1\alpha} \mathcal{H}_{\epsilon/h}$, which is analogous to how
$\mathcal{H}$ transformed when $U(q)=q^\alpha$, except for the introduction of 
the parameter $\epsilon$. We  then define $\Psi_\epsilon$ as the
solution to \eqref{eq:ideaPoison} with $\mathcal{H}$ replaced by
$\mathcal{H}_\epsilon$. Following the same logic as before, one sees
that $\Psi_\epsilon$ transforms to $h^{\frac12+\frac1\alpha}
\Psi_{\epsilon/h}$ under  $(P,Q) \mapsto (h^{\frac12} P,
h^{\frac1\alpha} Q)$. Similarly,  $ P\partial_P
  \Psi$ and  $\partial^2_P\Psi$  transform to $h^{\frac12 +
  \frac1\alpha} P\partial_P
  \Psi_{\epsilon/h}$ and
  $h^{\frac1\alpha-\frac12}\partial^2_P\Psi_{\epsilon/h}$,
  respectively. Hence we could repeat the same analysis if one had
  uniform control over the size of $\Psi_\epsilon$, $P\partial_P
  \Psi_{\epsilon}$ and $\partial^2_P\Psi_{\epsilon}$ as $\epsilon
  \rightarrow 0$. However, in all cases the rigorous extraction of the 
  needed scaling of the original $\Psi$ or this family of solutions $\Psi_\epsilon$, and in particular the scaling of their derivatives, seems elusive.  For this reason, we will modify the original PDE in \eqref{eq:ideaPoison}  by
introducing an approximate dynamics which will be asymptotically the same
as the dynamics driven by the Hamiltonian but which will scale exactly homogeneously
in the spirit of the previous paragraph. This will allow us to control
the needed terms but it will come with a cost. That is, the resulting
solution $\Psi$ will only be globally continuous and not globally $C^2$. It
will however be piecewise $C^2$ and the ideas from \cite{HerMat15i,
  HerMat15ii} will be exploited to nonetheless prove $H+\Psi$ is a
Lyapunov function for the time $t$ dynamics.

\subsubsection{The relationship to the ``$pq$'' trick and   Hypocoercivity}\label{sec:Hypocoercivity}
We now make a small digression and return to the ``trick'' used in the non-singular case of adding
$\gamma_0pq$ for some choice of positive $\gamma_0$ as discussed in
Section~\ref{Heuristics}.  In light of the construction used in this
paper, it is interesting to ask if $\gamma_0 p q$ is the solution of
an appropriate Poisson equation of the problem with a potential
$U(q)=q^{2n}/(2n)$, since this potential represents the behavior
at infinity of the class of potentials for which that construction is
used. We begin by observing that for the corresponding Liouville
operator $\mathcal{H}$ one has
\begin{align*}
  \mathcal{H}(pq) =  p^2 - q^{2n}= 
  (1+n  )p^2   - n p^2 -  q^{2n}=  (1+n)p^2   - 2n H(p,q).
\end{align*}
Hence multiplying by $\frac\gamma{1+n}$ and calculating that $\mathcal{A}( p^2)(q,p) = \frac{2n}{n+1}H(q,p)$, we see that $\Psi(q,p)= \frac{\gamma}{1+n} pq$ is a
solution to
\begin{align*}
  (\mathcal{H}\Psi)(q,p) = \gamma p^2 - \gamma\mathcal{A}( p^2)(q,p)\,.
\end{align*}
Hence this ``trick'' is exactly a
version of the ideas in this paper, namely solving the correct, asymptotically relevant Poisson
equation. It would be interesting to understand how this point of view
fits together with the ideas contained in the theory of hypocoercivity  as
developed by C. Villani \cite{Villani_2009} and subsequent authors \cite{2013arXiv1308.4938B, MR3324910,MR3413926,MR3522857}.

\subsubsection{The approximate dynamics}\label{ApproxDyn}
Rather than using the trajectories defined by the full Hamiltonian $H$
to build the corrector $\Psi$ via the method of characteristics, we
will use the trajectories defined by a ``piecewise Hamiltonian".  This
has the advantage of simplifying, yet capturing the dynamics at large
energies in various regions in the state space $\H$.  This, in
particular, will allow for easier analysis of our chosen corrector, as
the PDEs satisfied by $\Psi$ locally in various regions in $\H$ will
be far simpler than the equation \eqref{eq:ideaPoison} in property (I).

To introduce the approximate dynamics, recall that
\begin{align*}
H(Q,P)=\frac{P^2}{2}+U(Q)= \frac{P^2}{2}+ \sum_{i=1}^{l} a_i Q^{\alpha_i}
\end{align*}
where $\alpha_1 >2, a_1>0, a_l >0$, $\alpha_l <0$ and
\begin{align*}
\alpha_1 > \alpha_2 > \cdots > \alpha_l .  
\end{align*}
Because two parts in $U(Q)$ will play a special role throughout the rest of the paper, we let $\alpha_1= \alpha$, $a_1=a$, $\alpha_l=-\beta$, $a_l=b$ for simplicity.  For $(Q,P) \in \H$ let 
\begin{align}
\label{eqn:sham}
K(Q,P)= \frac{P^2}{2}+ b Q^{-\beta} \,\, \text{ and }\,\, J(Q,P)= \frac{P^2}{2}+ a Q^\alpha,  
\end{align}
and for $\xi_*, h_*>0$ define the following regions in the state space $\H$:     
\begin{align*}
\mathcal{S}_1(\xi_*, h_*)&= \{ (Q,P) \in \mathbb{H}\,: \, P^2 Q^{\beta}\leq \xi_*^2, \, Q< 1, K(Q,P) \geq k(h_*) \} \\
\mathcal{S}_2(\xi_*, h_*)&= \{ (Q,P) \in \mathbb{H} \, : \,  P^2 Q^{\beta}\geq  \xi_*^2,   P^2 Q^{-\alpha}\geq  \xi_*^2, \, H(Q,P) \geq h_*\}\\
\mathcal{S}_3(\xi_*, h_*)&= \{ (Q,P) \in \mathbb{H}  \, :  \, P^2 Q^{-\alpha} \leq \xi_*^2, \, Q>1, \, J(Q,P) \geq j(h_*)   \}
\end{align*}
where $k(h)$ and $j(h)$ are boundary functions to be introduced momentarily.  Both of the parameters $\xi_*, h_*$ should be thought of as large, and we will see soon that $k(h), j(h)\approx h$ for $h>0$ large.  The parameter $\xi_*>0$ will be increased at several instances throughout the paper.  Moreover, we will often choose the parameter $h_*$ to depend on $\xi_*$.  

To help motivate the regions above, observe that as $H(Q,P) \rightarrow \infty$ with $(Q,P)\in \mathcal{S}_1(\xi_*, h_*)$ we have  
\begin{align*}
H(Q,P) = Q^{-\beta} \bigg[ \frac{P^2 Q^\beta}{2} + b+o(1)\bigg] = K(Q,P) + Q^{-\beta} o(1)
\end{align*}
and as $H(Q,P)\rightarrow \infty$ with $(Q,P) \in \mathcal{S}_3(\xi_*, h_*)$
\begin{align*}
H(Q,P) = Q^\alpha \bigg[\frac{P^2 Q^{-\alpha}}{2} +a + o(1) \bigg]= J(Q,P) + Q^\alpha o(1).  
\end{align*}
Since $P^2 Q^\beta$ is bounded on $\mathcal{S}_1(\xi_*, h_*)$ and $P^2 Q^{-\alpha}$ is bounded on $\mathcal{S}_3(\xi_*, h_*)$, this calculation suggests that we should take the approximate dynamics in $\mathcal{S}_1(\xi_*, h_*)$ to be the dynamics determined by the Hamiltonian $K(Q,P)$.  
 Similarly in $\mathcal{S}_3(\xi_*, h_*)$, we should take the approximate dynamics to be the dynamics determined by the Hamiltonian $J(Q,P)$.  The region $\mathcal{S}_2(\xi_*, h_*)$ corresponds to an asymptotically insignificant piece of the dynamics at large energies when $\xi_*$ is also large, and therefore should serve merely as a ``transition zone" between two other regimes, $\mathcal{S}_1(\xi_*, h_*)$ and $\mathcal{S}_3(\xi_*, h_*)$.  This, in particular, suggests that we maintain the dynamics determined by $H$ in the region $\mathcal{S}_2(\xi_*, h_*)$.

 \begin{remark}\label{rem:Scaling}
   It is also instructive to understand how the analogous regions for
   $\mathcal{H}_\epsilon$ transform under the scaling $(P,Q)\mapsto
(h^{\frac12}P,h^{\frac1\alpha}Q)$. If in the regions $\mathcal{S}_i$ we replace $P^2Q^\beta$ by
$P^2Q^\beta\epsilon^{-1-\frac\beta\alpha}$, this then defines
correct regions $\mathcal{S}_i^\epsilon$ corresponding to 
$\mathcal{H}_\epsilon$ (ignoring the truncation for small $H$ for the moment).
Notice that the boundary between
$\mathcal{S}_3^\epsilon$ and $\mathcal{S}_2^\epsilon$ would remain unchanged as $\epsilon
\rightarrow 0$ yet the boundary between $\mathcal{S}_1^\epsilon$ and
$\mathcal{S}_2^\epsilon$ will collapse
towards the $q=0$ axis. Hence as $\epsilon \rightarrow0$, the region
$\mathcal{S}_1^\epsilon$ becomes a vanishingly small part of the phase
space. Furthermore by making $\xi_*$ large we can decrease the
importance of the dynamics in $\mathcal{S}_2^\epsilon$ by making this region smaller. Thus we expect only 
the dynamics in region $\mathcal{S}_3^\epsilon$ to be relevant asymptotically. In  $\mathcal{S}_3^\epsilon$,
the potential $U_\epsilon(q)$ is dominated by $aQ^\alpha$ as $\epsilon
\rightarrow 0$ uniformly and we expect the dynamics governed by the
Hamiltonian $J$ defined above to dominate. We will see that all of
these predictions hold and that they are behind all of the
construction on which we now embark. 
 \end{remark}

To define the approximate dynamics precisely, we need some additional notation.  For $h_0=h_0(\xi_*)>0$ large enough and $h\geq h_0$, let $(Q_1,P_1)=(Q_1(\xi_*, h),P_1(\xi_*, h)) \in \mathcal{S}_1(\xi_*, h_0)$ satisfy
\begin{align*}
H(Q_1, P_1)&= h, \,\, P_1^2=\xi_*^2 Q_1^{-\beta}, \,\,  P_1 >0  
\end{align*}  
and $(Q_3, P_3)= (Q_3(\xi_*, h), P_3(\xi_*, h)) \in \mathcal{S}_3(\xi_*, h_0)$ satisfy
\begin{align*}
H(Q_3, P_3)&= h, \,\, P_3^2=\xi_*^2 Q_3^{\alpha}, \,\,  P_3 >0.  
\end{align*} 
From the asymptotic observations made above, we note that as $h\rightarrow \infty$
\begin{align*}
Q_1^{-\beta}\bigg[ \frac{\xi^2_*}{2}+ b +o(1)\bigg]= h\,\, \text{ and }\,\, Q_3^\alpha \bigg[ \frac{\xi_*^2}{2}+ a +o(1)\bigg] =h.  
\end{align*}
Now, for $h\geq h_0$, $h_0=h_0(\xi_*)>0$ large enough, define $k(h), j(h)>0$ by
\begin{align*}
k(h) = h -  \sum_{i=1}^{l-1} a_i Q_1^{\alpha_i}, \qquad j(h)= h - \sum_{i=2}^l a_i Q_3^{\alpha_i}  
\end{align*}
and notice that 
\begin{align*}
\lim_{h\rightarrow \infty} h^{-1}k(h) = \lim_{h\rightarrow \infty} h^{-1} j(h) = 1.    
\end{align*}
By perhaps again increasing $h_0$ if necessary, also observe that for all $i\neq j$
\begin{align*}
\text{interior}(\mathcal{S}_i(\xi_*, h_0)) \cap \text{interior} (\mathcal{S}_j(\xi_* ,  h_0)) = \emptyset.  
\end{align*} 
Setting 
\begin{align*}
\mathcal{S}_i^+ (\xi_*, h_0) &= \mathcal{S}_i (\xi_*, h_0) \cap \{(Q,P) \in \H \,: \, P\geq 0 \}\\
\mathcal{S}_i^{-}(\xi_*, h_0)&= \mathcal{S}_i(\xi_*, h_0) \cap \{(Q,P) \in \H \,: \, P\leq 0 \},
\end{align*}
with this choice of $h_0$ we have sketched the regions $S_j(\xi_*, h_0)$ in Figure~\ref{fig:regionplot}.

\begin{figure}
\centering
\begin{tikzpicture}[xscale=.5, yscale=1.6]

\draw[thick,color=black,opacity=.75] (-11,-.5) -- (11,-.5) node[right] {$P$};
\draw[thick, color=black,opacity=.75] (0, .-.4) -- (0, -.6)
node[below]{$0$};
\draw[thick, color=black,opacity=.75] (-10.6, -.2) -- (-11.4, -.2)
node[left]{$0$};
\draw[dashed] (-11, -.2) --(11, -.2);

\draw[dashed] (0,-.4)--(0,2.5);

\draw[thick,color = black,opacity=.75] (-11,-.5) -- (-11,2.5) node[left] {$Q$};


\draw[thick, red, smooth,samples=100,domain=.3535:1.675] plot({(40-2*\x^4-2/\x^2)^(1/2)}, {\x});
\draw[thick, red, smooth,samples=100,domain=.3535:1.675] plot({-(40-2*\x^4-2/\x^2)^(1/2)}, {\x});
\draw[thick, blue, smooth,samples=100,domain=1.675:2.1063] plot({(39.357-2*\x^4)^(1/2)}, {\x});
\draw[thick, blue, smooth,samples=100,domain=1.675:2.1063] plot({-(39.357-2*\x^4)^(1/2)}, {\x});
\draw[thick, green, smooth,samples=100,domain=.2238:0.3535] plot({(40.012-2/\x^2)^(1/2)}, {\x});
\draw[thick, green, smooth,samples=100,domain=.2238:.3535] plot({-(40.012-2/\x^2)^(1/2)}, {\x});
\draw[thick, smooth,samples=100,domain=.16:.3585] plot({-3^(1/2)/\x}, {\x});
\draw[thick, smooth,samples=100,domain=.16:.3585] plot({3^(1/2)/\x}, {\x});
\draw[thick, smooth,samples=100,domain=1.675:2.5] plot({3^(1/2)*\x^2}, {\x});
\draw[thick, smooth,samples=100,domain=1.675:2.5] plot({-3^(1/2)*\x^2}, {\x});


\node at (-1,2.3) {$\mathcal{S}_3^-$};

\node at (-5,2.3) {$\Psi_3^- + \tfrac12{\overline{\Psi}_1^-} + \overline{\Psi}_2^-$};

\node at (5, 2) {$\mathcal{S}_3^+$};

\node at (2.2, 2.3) {$\Psi_3^+ - \tfrac12{\overline{\Psi}_3^+}$};

\node at (7,.6) {$\mathcal{S}_2^+$};

\node at (8,1.5) {$\Psi_2^+  +  \tfrac12{\overline{\Psi}_3^+}$};

\node at (-6.5,1.6) {$\mathcal{S}_2^-$};

\node at (-8,.8) {$\Psi_2^- + \tfrac12{\overline{\Psi}_1^-}$};

\node at (-7,.03) {$\mathcal{S}_1^-$};

\node at (-3,.03) {$\Psi_1^- - \tfrac12{\overline{\Psi}_1^-} $};

\node at (1.1, .04) {$\mathcal{S}_1^+$};

\node at (5, .04) {$\Psi_1^+ + \overline{\Psi}_2^+  + \tfrac12{\overline{\Psi}_3^+}$};

\end{tikzpicture}
\caption{The regions $\mathcal{S}_i$, $i=1,2,3,4$, are plotted above along with the form of $\Psi$ in each region.   The rotation
  along a cycle $\Gamma(h)$ for the approximate dynamics is in the
  counterclockwise direction.  Thus boundary contributions $\overline{\Psi}_i^\pm$ accumulate in the clockwise direction.  The specific choice of $\Psi$ was made so that $\Psi(Q,0)=0$ to exploit the symmetry in the problem.  Also note that, in light of
  Remark~\ref{rem:Scaling}, we expect for large energy and large
  $\xi_*$ that only the dynamics in the region $\mathcal{S}_3$
  will be relevant.    
}
\label{fig:regionplot}
\end{figure}
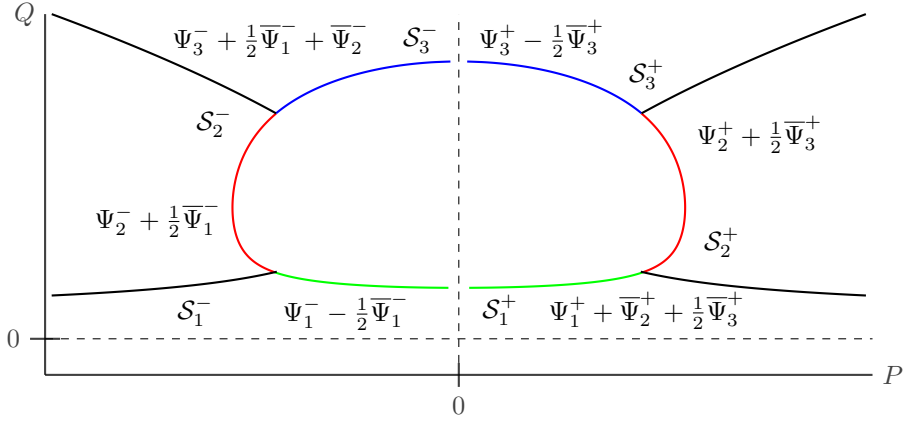

We can now define the approximate dynamics.  For simplicity, set $\mathcal{S}_i=\mathcal{S}_i(\xi_*, h_0)$, $\mathcal{S}_i^\pm= \mathcal{S}_i^\pm(\xi_*, h_0)$ and $\H_{h_0}= \bigcup \mathcal{S}_i$.    

\begin{definition}[The Approximate Dynamics]
\label{def:approxdyn}
For $(Q,P) \in \H_{h_0}$, the \emph{approximate dynamics started from} $(Q,P)$ is the solution of the differential equation
\begin{align*}
(Q_t, P_t)=(Q,P) + \int_0^t X(Q_s, P_s) \, ds
\end{align*}
where $X: \H_{h_0} \rightarrow \R^2$ is given by 
\begin{align*}
X(Q,P) = \begin{cases}
(P, b\beta Q^{-\beta-1}) & \text{ if } \, (Q,P) \in \mathcal{S}_1 \setminus  \mathcal{S}_2^+ \\
(P, -U'(Q)) & \text{ if } \, (Q,P) \in \mathcal{S}_2^+ \setminus \mathcal{S}_3\\
(P, -a\alpha Q^{\alpha-1})& \text{ if } (Q,P) \in \mathcal{S}_3\setminus \mathcal{S}_2^-  \\
(P, -U'(Q)) & \text{ if } \, (Q,P) \in \mathcal{S}_2^- \setminus \mathcal{S}_1 \end{cases}.  
\end{align*}
\end{definition}

One can check that for initial conditions $(Q,P) \in \H_{h_0}$, the approximate dynamics started from $(Q,P)$ has a unique solution with a corresponding continuous solution curve $\Gamma(h)$, where $H(Q,P)=h$, given by the union of the following curves
\begin{align*}
\Gamma_1(h) &= \{(Q,P) \in \mathcal{S}_1\,: \, K(Q,P)=k(h) \}\\
\Gamma_2(h) & =  \{(Q,P) \in \mathcal{S}_2\,: \, H(Q,P)=h \}\\
\Gamma_3(h)&=  \{(Q,P) \in \mathcal{S}_3\,: \, J(Q,P)=j(h) \}.
\end{align*}

\subsubsection{Poisson equations and $\Psi$}
\label{sec:PoisonEquations}
Using the approximate dynamics, we will now define the corrector
$\Psi$.  We begin by defining the transport operators corresponding to
flow generated by the approximate dynamics defined above. In other
words, they are the first order differential operators whose
characteristics correspond to the approximate dynamics. 
Defining the operators $\mathcal{K}$ and $\mathcal{J}$ by
\begin{align*}
\mathcal{K}&\eqdef P\partial_Q + \beta b Q^{-\beta-1} \partial_Q\\  
\mathcal{J}&\eqdef P\partial_Q - \alpha a Q^{\alpha-1} \partial_Q.  
\end{align*}
we see that transport generated by the operator $\mathcal{K}$
corresponds to the flow of the approximate dynamics in
$\mathcal{S}_1\setminus \mathcal{S}_1 \cap \mathcal{S}_2^+$ while the
transport generated by the operator $\mathcal{J}$ corresponds to the
approximate dynamics in
$\mathcal{S}_3\setminus \mathcal{S}_2^- \cap \mathcal{S}_3$.  We recall
that in the remaining regions in $\H_{h_0}$, the dynamics is that determined by the
full Liouville operator $\mathcal{H}$.

For $h\geq h_0$ and $l=1,2,3$, we let $G_l(h)$ denote the total time spent by the approximate dynamics in $\mathcal{S}_l$ during one complete cycle on $\Gamma(h)$, and define $T(h)= \sum_{j=1}^3 G_l(h) .$  For $l=1,2,3$ and $h\geq h_0$, we let $F_l(h)$ be given by  
\begin{align*}
F_l(h) &\eqdef \int_0^{T(h)}  \mathbf{1}_{\mathcal{S}_l}(Q_s, P_s) P_s^2 \, ds   
\end{align*}
where in the above $\mathbf{1}_{\mathcal{S}_l}$ denotes the indicator function on $\mathcal{S}_l$, $(Q_s, P_s)$ corresponds to the coordinates of the approximate dynamics, and we are taking as our initial condition any point $(Q_0,P_0)$ belonging to $\Gamma(h)$.  For positive parameters $c_l^+$ and $c_l^-$, $l=1,2,3$, and $h\geq h_0$, we define the weighted averages $\mathcal{A}_+(h)$ and $\mathcal{A}_-(h)$ by  
\begin{align}
\label{eqn:defApm}
\mathcal{A}_\pm(h) \eqdef\frac{\sum_{i=1}^3 c_i^\pm F_i(h)}{\sum_{i=1}^3 c_i^\pm G_i(h)}.  
\end{align}

\begin{remark}
Observe that $\mathcal{A}_\pm(h)$ are slight modifications of the average $\mathcal{A}_1(P^2)(h)$ of $P^2$ over one cycle of the deterministic dynamics defined by the full Hamiltonian $H$.  More precisely, they are weighted versions (with weights $c_l^\pm$) of the average of $P^2$ over one cycle of the approximate dynamics.  Later we will see that for every $\epsilon >0$ there exists $\xi_*>0$ large enough such that for all $h\geq h_0=h_0(\xi_*)$ large enough
\begin{align*}
1-\epsilon \leq \frac{\mathcal{A}_1(P^2)(h)}{\mathcal{A}_\pm(h)} \leq 1+\epsilon .   
\end{align*}   
More specifically, we will see that the asymptotically dominant part of $\mathcal{A}_\pm(h)$ is $c_3^\pm F_3(h)/c_3^\pm G_3(h)=F_3(h)/G_3(h)$; that is, the dominant contribution to the dissipation at large energies comes from region $\mathcal{S}_3$. We will need these slight modifications and the parameters $c_l^\pm$ to ensure that $\Psi$ defined below is smooth enough to apply Peskir's extension of It\^{o}'s formula~\cite{Peskir_07} and to deal with the signs of the local time contributions in $d\Psi(q_t,p_t)$ arising because $\Psi$ will not quite be globally $C^2$.                    
\end{remark}

Just like the original dynamics determined by the full Liouville operator $\mathcal{H}$, the function $\Psi$ will be broken into several pieces.  To introduce them, first recall the definitions of $j(h)$ and $k(h)$ introduced after Remark~\ref{rem:Scaling} and note that, by increasing $h_0$ if necessary, the functions $$j:[h_0, \infty) \rightarrow [j(h_0), \infty), \, \, k:[h_0, \infty) \rightarrow [k(h_0), \infty)$$ are twice continuously differentiable with twice continuously differentiable inverse functions $$j^{-1}: [j(h_0), \infty) \rightarrow [h_0, \infty),\,\, k^{-1}: [k(h_0), \infty) \rightarrow [h_0, \infty).$$  Moreover, it can be shown by implicit differentiation of $Q_1(h)$ and $Q_3(h)$ with respect to $h$ that the inverse functions satisfy
\begin{align}
\label{eqn:jk1}
\lim_{h\rightarrow \infty}  (h^{-1} j^{-1}(h)) &= \lim_{h\rightarrow \infty}  (h^{-1} k^{-1}(h))= 1 \\
\label{eqn:jk2}\lim_{h\rightarrow \infty} (j^{-1})'(h) &= \lim_{h\rightarrow \infty} (k^{-1})'(h) =1\\
\label{eqn:jk3}\lim_{h\rightarrow \infty} |h(j^{-1})''(h)| &=  \lim_{h\rightarrow \infty} |h(k^{-1})''(h)|=0.    
\end{align}

We let $\Psi_1^+$ and $\Psi_2^-$ be defined on $\mathcal{S}_1$ as the solutions of the following boundary-value PDEs
\begin{align}
\label{eqn:pois1}
\begin{cases}
&(\mathcal{K} \Psi_1^\pm)(Q,P) = \gamma c_1^\pm\big( P^2-  \mathcal{A}_\pm (k^{-1}(K))\big)\\
&\Psi_1^\pm(Q,P)=0 \,\, \text{ for } P^2 Q^{\beta}= \xi_*^2 , \, P>0 
\end{cases}
\end{align}
where $K=K(Q,P)$ is as in~\eqref{eqn:sham}.  Define $\Psi_2^+$ and $\Psi_2^-$ on $\mathcal{S}_2$ by
\begin{align}
\label{eqn:pois4}
\begin{cases}
&(\mathcal{H} \Psi_2^+)(Q,P) = \gamma c_2^+ \big( P^2- \mathcal{A}_{+}(H)\big)\\
&\Psi_2^+(Q,P)=0 \,\, \text{ for } P^2 Q^{-\alpha}= \xi_*^2 , P>0.  
\end{cases}
\end{align}
and 
\begin{align}
\begin{cases}
&(\mathcal{H} \Psi_2^-)(Q,P) = \gamma c_2^- \big( P^2-  \mathcal{A}_{-}(H)\big)\\
&\Psi_2^- (Q,P)=0 \,\, \text{ for } P^2 Q^{\beta}= \xi_*^2 , P<0
\end{cases}.
\end{align}
Lastly, for $(Q,P) \in \mathcal{S}_3$ define $\Psi_3^+$ and $\Psi_3^-$ as the solutions of
\begin{align}
\begin{cases}
\label{eqn:pois4}
&(\mathcal{J} \Psi_3^\pm)(Q,P) = \gamma c_3^\pm \big( P^2- \mathcal{A}_\pm(j^{-1}(J))\big)\\
&\Psi_3^\pm(Q,P)=0 \,\, \text{ for } P^2 Q^{-\alpha}= \xi_*^2 , P<0
\end{cases}
\end{align}where $J=J(Q,P)$ is as in~\eqref{eqn:sham}.

\begin{remark}
At this point it is helpful to compare the righthand sides of the
equations above with the righthand side of the equation
\eqref{eq:ideaPoison} in property (I).  Because $k^{-1}(K)$ and $j^{-1}(J)$ are asymptotically equivalent to $H$ in, respectively, $\mathcal{S}_1$ and $\mathcal{S}_3$, the only noticeable difference between the two is the presence of the parameters $c_l^\pm$.  However as we will see later, we will be able to choose $c_l^\pm \leq 1$, $i=1,2,3$, arbitrarily close to $1$.  Therefore, due to the asymptotic formula for $\mathcal{A}_\pm$ discussed in the previous remark, the equations satisfied by $\Psi_i^\pm$ approximate, up to a small constant, the equation in (I) when $H\rightarrow \infty$.  We will see that this constant can be made arbitrarily small by first picking the boundary parameter $\xi_*>0$ large enough.          
\end{remark}

Because we have defined $\Psi_i^\pm$ using zero boundary conditions and $\Psi_i^\pm \neq 0$ on the other boundary in its region of definition, we cannot (as may be suggested by the above) by fixing a $+$ or $-$ define our corrector $\Psi$ to simply be $\Psi_i^\pm$ on $\mathcal{S}_i$.  In particular, although we will see that each $\Psi_i^\pm$ is $C^2$ on $\mathcal{S}_i^\pm$, such a choice would mean that $\Psi$ is not globally continuous.  

To see how to obtain the desired global continuity, let $g_l(Q,P)$, $l=1,2,3$, be the first exit time of the approximate dynamics from $\mathcal{S}_l$ started from $(Q,P) \in \mathcal{S}_l(\xi_*, h_0)$ and for $(Q,P) \in \mathcal{S}_l$ define 
\begin{align*}
f_l(Q,P)\eqdef \int_0^{g_l(Q,P)} P_s^2 \,ds 
\end{align*} 
where again we recall that $P_s$ is the momentum coordinate of the approximate dynamics.  Applying the method of characteristics to solve equations~\eqref{eqn:pois1}-\eqref{eqn:pois4} produces the following expressions for $\Psi_i^\pm(Q,P)$:
\begin{align*}
\Psi_1^\pm (Q,P)&= \gamma c_1^\pm \big( \mathcal{A}_\pm (k^{-1}(K)) g_1(Q,P)-f_1(Q,P)\big)\\
\Psi_2^\pm(Q,P)&= \gamma c_2^\pm \big( \mathcal{A}_\pm(H) g_2(Q,P)-f_2(Q,P)\big)\\
\Psi_3^\pm(Q,P)&= \gamma c_3^\pm \big( \mathcal{A}_\pm(j^{-1}(J)) g_3(Q,P)-f_3(Q,P)\big)
\end{align*}
where  $H=H(Q,P)=\frac{P^2}{2}+ U(Q)$, $K=K(Q,P)= \frac{P^2}{2}+ b Q^{-\beta}$ and $J=J(Q,P)=\frac{P^2}{2}+ a Q^{\alpha}$.  Hence for $(Q,P)\in \mathcal{S}_i\cap \Gamma(h)$, the value of $\Psi_i^\pm(Q,P)$, $i=1,2,3,4$, on the boundary where it is nonzero is given by 
 \begin{align}
\nonumber \overline{\Psi}_1^\pm(h)&\eqdef \gamma c_1^\pm \big( \mathcal{A}_\pm(h) G_1(h)-F_1(h)\big)\\
\label{eqn:psi2bar}\overline{\Psi}_2^\pm (h)&\eqdef \frac{\gamma}{2} c_2^\pm \big( \mathcal{A}_\pm(h) G_2(h)-F_2(h)\big)\\
\nonumber\overline{\Psi}_3^\pm(h)&\eqdef \gamma c_3^\pm \big( \mathcal{A}_\pm(h) G_3(h)-F_3(h)\big).
\end{align}

\begin{remark}
Recall that for $h\geq h_0$, $G_l(h)$ denotes the total time spent by the approximate dynamics in $\mathcal{S}_l$ during one complete cycle on $\Gamma(h)$, and
\begin{align*}
F_l(h) = \int_0^{T(h)} \mathbf{1}_{\mathcal{S}_l}(Q_s, P_s) P_s^2 \, ds  
\end{align*}
where $T(h) = G_1(h) + G_2(h) + G_3(h)$ is the time to complete one cycle.  Hence the factor of $\frac{1}{2}$ appears on the righthand side of the expression for $\overline{\Psi}_2^\pm(h)$ above by symmetry since only one half of the trajectory in $\mathcal{S}_2$ is traversed starting in either $\mathcal{S}_2^+$ or $\mathcal{S}_2^-$ upon exiting the domain.  See Figure~\ref{fig:regionplot}.       
\end{remark}

Finally to define $\Psi$, let $\psi\in C^\infty(\R: [0,1])$ satisfy $\psi(x)=0$ for $x\leq 2h_0$ and $\psi(x)=1$ for $x\geq 3h_0$.  

\begin{definition}[Definition of $\Psi$] 
\label{def:PSI}
For $(Q,P)\in \H_{h_0}$ and $h=H(Q,P)$, define 
\begin{align*}
\widetilde{\Psi}(Q,P) = \begin{cases}
\Psi_1^+(Q,P) + \overline{\Psi}^+_2(h) + \frac{1}{2} \overline{\Psi}_3^+(h) & \text{ if } (Q,P) \in \mathcal{S}_1^+\\
\Psi_2^+(Q,P) + \frac{1}{2} \overline{\Psi}_3^+(h) & \text{ if } (Q,P) \in \mathcal{S}_2^+\\
\Psi_3^+(Q,P) - \frac{1}{2}\overline{\Psi}_3^+(h)  & \text{ if } (Q,P) \in \mathcal{S}_3^+\\
\Psi_1^-(Q,P) - \frac{1}{2} \overline{\Psi}^+_1(h) & \text{ if } (Q,P) \in \mathcal{S}_1^-\\
\Psi_2^-(Q,P) + \frac{1}{2} \overline{\Psi}_1^-(h) & \text{ if } (Q,P) \in \mathcal{S}_2^-\\
\Psi_3^-(Q,P) + \frac{1}{2}\overline{\Psi}_1^-(h) + \overline{\Psi}_2^-(h)  & \text{ if } (Q,P) \in \mathcal{S}_3^-
\end{cases}
\end{align*}
For $(Q,P)\in \H\setminus \H_{h_0}$, we define $\widetilde{\Psi}(Q,P)=0$.  The function $\Psi:\H \rightarrow \R$ is defined by 
\begin{align*}
\Psi(Q,P)=
\psi(h)\widetilde{\Psi}(Q,P)  
\end{align*}
where $h= H(Q,P)$.  By increasing $h_0$ if necessary, $\Psi$ is continuous everywhere and satisfies $\Psi(Q,0)=0$. See Remark~\ref{rem:cont-smmoth} for
further elaboration.
\end{definition}

As a visual aid for the reader, we have provided Figure~\ref{fig:regionplot} which plots the regions and gives the form of $\Psi$ in each region.

\begin{remark}\label{rem:cont-smmoth}
  In the Appendix, we will see easily by inspection of the formulas
  derived there that $\Psi_i^\pm(Q,P)$ is $C^2$ on
  $\mathcal{S}_i^\pm$ and that $\overline{\Psi}_i^\pm(h)$ is $C^2$
  for $h\geq h_0$.  In particular, $\Psi$ is $C^2$ everywhere
  EXCEPT along the neighboring curves dividing the regions
  $\mathcal{S}_i^\pm$.  In fact if we show that $\Psi$ is globally continuous, we may apply the generalized It\^{o} formula due to Peskir~\cite{Peskir_07}, giving the existence of the It\^{o} differential $d\Psi(q_t, p_t)$.  
  
  To see that $\Psi$ is continuous along these neighboring curves, it is helpful to consider the diagram in Figure~\ref{fig:regionplot} which gives the definition of $\Psi$ in each region.  First observe that since $\Psi_1^+=0$ and $\Psi_2^+= \overline{\Psi}_2^+$ on the boundary $\mathcal{S}_1^+\cap \mathcal{S}_2^+$, we find that for $(q, p) \in \mathcal{S}_1^+ \cap \mathcal{S}_2^+$
  \begin{align*}
  \lim_{\substack{(Q,P) \rightarrow (q, p)\\ (Q,P) \in \mathcal{S}_1^+}} \Psi(Q,P) =  \lim_{\substack{(Q,P) \rightarrow (q, p)\\ (Q,P) \in \mathcal{S}_2^+}} \Psi(Q,P) = \overline{\Psi}_2^+(h) + \frac{1}{2}\overline{\Psi}_3^+(h).  
  \end{align*}  
  where $h= H(q, p)$.  Similar observations will show that $\Psi$ is continuous along the boundaries $\mathcal{S}_2^+ \cap \mathcal{S}_3^+$, $\mathcal{S}_1^- \cap \mathcal{S}_2^-$ and $\mathcal{S}_2^- \cap \mathcal{S}_3^-$.  This leaves us to check that $\Psi$ is continuous at $P=0$.  To see this, first observe that by using the formulas~\eqref{eqn:psi2bar} and~\eqref{eqn:defApm}, for $h\geq h_0$ 
\begin{align*}
\frac{1}{2}\overline{\Psi}_1^{\pm}(h) + \overline{\Psi}_2^{\pm}(h) + \frac{1}{2}\overline{\Psi}_3^{\pm}(h) &= \frac{1}{2}\gamma \sum_{i=1}^3 \mathcal{A}_{\pm}(h) c_i^\pm G_i(h) -c_i^\pm F_i(h)\\
&=\frac{1}{2}\gamma \sum_{i=1}^3 (c_i^\pm F_i(h) - c_i^\pm F_i(h))=0.  
\end{align*}
For $(q,p) \in \mathcal{S}_1 \cap \{ (Q,P) \in \H \, : \, P=0\}$ and $h= H(q, p)$, we find that 
\begin{align*}
\Psi_1^-(q, p) - \frac{1}{2}\overline{\Psi}^-_1(h) =  \frac{1}{2}\overline{\Psi}^-_1(h) -  \frac{1}{2}\overline{\Psi}^-_1(h)=0
 \end{align*}
 and 
 \begin{align*}
 \Psi_1^+(q, p) + \overline{\Psi}_2^+(h) + \frac{1}{2}\overline{\Psi}^+_3(h) = \frac{1}{2}\overline{\Psi}_1^{+}(h) + \overline{\Psi}_2^{+}(h) + \frac{1}{2}\overline{\Psi}_3^{+}(h)=0.  
 \end{align*}
 This now implies continuity of $\Psi$ on $\mathcal{S}_1 \cap \{ (Q,P) \in \H \, : \, P=0\}$.  A similar calculation shows that $\Psi$ is continuous on $\mathcal{S}_3 \cap \{ (Q,P) \in \H \, : \, P=0\}.$  \end{remark}


\section{Consequences of Lyapunov structure}\label{sec:consequences}

In this section, we reduce the proof of the main theorem, Theorem~\ref{MainThm}, to the proofs of Theorem~\ref{thm:Lyap} and Theorem~\ref{thm:LyapPr} below.  As we will see in the following section, both theorems will be immediate consequences of Lemma~\ref{lem:Lyappsi}, a result encapsulating the needed properties of the corrector $\Psi$ as given in Definition~\ref{def:PSI}.  

\begin{theorem}\label{thm:Lyap}
Let $\epsilon >0$.  Then there exists $\xi_*>0$ and $h_0(\xi_*)>0$ large enough such that the functions $\Psi$ and $V= H+ \Psi$ satisfy the following:
   \begin{itemize}
  \item[(a)]  As $H(q,p) \rightarrow \infty$, $\Psi(q,p) = o(H(q,p)).$    
   \item[(b)] The It\^{o} differential of $\Psi(q_t, p_t)$ exists.  Furthermore, there exists a constant $C>0$ such that    \begin{align}
   \label{eqn:basicgeo}
  d V(q_t,p_t) &\leq -\gamma( \rate_*-\epsilon) V(q_t,p_t) \, dt + C dt +  dM(t)   
  \end{align}
for some $L^2$-martingale $M(t)$ with quadratic variation $\llangle M\rrangle_t$ satisfying 
\begin{align*}
\llangle M \rrangle_t = \sigma^2 \int_0^t  p_s^2  \, ds +  \int_0^t \Sigma(q_s, p_s) \, ds 
\end{align*}     
where $\Sigma: \H \rightarrow \R$ is locally bounded, measurable with $\Sigma(q,p) =o(H(q,p))$ as $H\rightarrow \infty$. 
  \end{itemize}
     \end{theorem}

\begin{remark}
The usage of the boundary parameters $\xi_*$ and $h_0=h_0(\xi_*)$ in the statement above allows us to tune the corrector $\Psi$ so as to get close to the predicted large energy dissipation constant $\gamma \Lambda_*$ for the Hamiltonian $H$, as discussed heuristically and numerically in Section~\ref{LyapOverview}.   
\end{remark}

Theorem~\ref{thm:Lyap} has the following immediate corollaries.
 
\begin{corollary}\label{globalExist}
Let $\tau_{\H}$ be the first exit time of $(q_t, p_t)$ from $\H$.  Then for all initial conditions $(p_0,q_0) \in \H$, $\tau_\H = \infty$ almost surely. Hence the
  local in time solutions to \eqref{sde1} for $(p_0,q_0) \in \H$
  provided by the standard theory are in fact global in time solutions
  contained in $\H$ for all time with probability one.
\end{corollary}
\begin{proof}[Proof of Corollary~\ref{globalExist}]
 See, for example, Theorem~2.1 of~\cite{MT_93}.  
 \end{proof}

For the next corollary, we momentarily return to considering the
unreduced system $(\mathbf{p}_t,\mathbf{q}_t)=(q_1(t),q_2(t),p_1(t),p_2(t))$ defined by \eqref{SDE4D}.

\begin{corollary}\label{S-existance}
  Let $\tau_\mathbb{S}$ be the first exit time of $(\mathbf{q}_t, \mathbf{p}_t)$ from $\mathbb{S}$; then for all initial conditions $(\mathbf{q}_0,\mathbf{p}_0) \in \mathbb{S}$, $\tau_{\mathbb{S}} =
  \infty$ almost surely. And hence the local in time solutions to
 \eqref{SDE4D} for $(\mathbf{q}_0,\mathbf{p}_0) \in \mathbb{S}$ provided by the standard theory
  are in fact global in time solutions contained in $\mathbb{S}$ for all time with
  probability one.
\end{corollary}

\begin{proof}[Proof of Corollary~\ref{S-existance}]
The existence of a global solution $(\mathbf{q}_t,\mathbf{p}_t)$  to
\eqref{SDE4D} is equivalent the existence of a global solution
$(\bar{q}_t,\bar{p}_t,\tilde q_t,\tilde p_t)$ which solves
\eqref{eqSystem}. The existence of a global solution to
$(\bar{q}_t,\bar{p}_t)$  follows directly from
Corollary~\ref{globalExist}. Since the pair $(\tilde q_t,\tilde
p_t)$ is independent of $(\bar{q}_t,\bar{p}_t)$, we can consider it
alone. Since it has no singularity, the existence of a global solution  for
$(\bar{q}_t,\bar{p}_t)$ can be found in many places including \cite{MattinglyStuartHigham02}.
\end{proof}

\begin{remark}
To assure that each of the dynamics above is well-defined for all finite times, it is sufficient to take $V=H$.  Indeed, using stopping times we can obtain the following bound from~\eqref{Heq}
\begin{align*}
\E_{(q,p)} H(q_t, p_t) \leq H(q,p) + \frac{\sigma^2}{2}t  
\end{align*}
for all times $t\geq 0$.  However, to obtain the stronger estimate~\eqref{eqn:basicgeo}, which highlights and is in agreement with the hueristic considerations of Section~\ref{LyapOverview} (see also equation~\eqref{HeqApprox}), we need the perturbation $\Psi$.  Furthermore, using the corrector $\Psi$ will allow us to conclude geometric ergodicity below as stated in the main result Theorem~\ref{MainThm}.       
\end{remark}

With the approproate Doeblin minorization condition (see Lemma~\ref{smallset} below), Theorem~\ref{thm:Lyap} implies geometric ergodicity of the process $(q_t, p_t)$, but in a much weaker weighted norm than used in the statement of Theorem~\ref{MainThm} (see Theorem~1.3 in~\cite{HairerMattingly08}).  The natural strategy employed to improve the norm of convergence is to exponentiate the existing Lyapunov function $V$ with a tuning parameter $c>0$; that is, now consider the test function
\begin{align*}
V_1(q,p) = \exp( c V(q,p)) = \exp( c(H(q,p) + \Psi(q,p))).   
\end{align*}      
Assuming for simplicity of discussion that $\Psi$ is globally $C^2$, we would then find that by construction 
\begin{align}
\nonumber \mathcal{L}V_1(q, p) &= cV_1(q, p)[ \mathcal{L} H(q, p) + \mathcal{L} \Psi(q, p) + c \gamma T (p+ \partial_p \Psi(q,p))^2 ] \\
\label{eqn:expVqv}&= c V_1(q,p)[- \gamma p^2 + \mathcal{L}\Psi(q,p) + c \gamma T p^2 + o(H(q,p)) ]\\
\nonumber &\leq c V_1(q,p)[-\gamma(\Lambda_* - \epsilon)  H(q,p) + c\gamma T p^2 + o(H(q,p)) ]
\end{align}
where $\epsilon >0$ is a small parameter which can be adjusted by tuning the boundary parameters in the definition of $\Psi$.  Upon making the ``brutal" bound $p^2 \leq 2 H(q,p)$ and picking $c\in (0, \Lambda_*/2T)$ and $\epsilon < (0, \Lambda_*-2cT)$, the estimate above becomes
\begin{align*}
\mathcal{L} V_1(q,p) & \leq c V_1(q,p) [ - \gamma( \Lambda_* - \epsilon - 2c T)H(q,p) + o(H(q,p))]
\end{align*}  
which then implies Theorem~\ref{MainThm} with $c\in (0, \Lambda_*/2T)$.  Recalling the definition of $\Lambda_*$ given in equation~\eqref{rateStar}, we note that $\Lambda_* \in (1,2)$ is fixed, so we do not quite realize the upper threshold of $1/T>\Lambda_*/2T$ for the constant $c$ given in the statement of Theorem~\ref{MainThm}.  Nevertheless, we should expect to be able to arrive at the threshold of $1/T$ since $\exp(c H(q,p))$ is integrable with respect to the unique invariant measure (see equation~\eqref{eqn:gibbs}) if and only if $c<1/T$.  

To see why this approach is not optimal in this way as well as how to fix it, recall that the lower-order perturbation $\Psi$ was constructed to exchange $-\gamma p^2$ with its average over one cycle of the approximate dynamics, thus leading to a globally dissipative Lyapunov functional of the form $V=H+ \Psi$.  However, when $V$ is exponentiated as above an additional quadratic variation term, namely $c\gamma T p^2$, arises (see equation~\eqref{eqn:expVqv}).  Thus, instead of correcting for $-\gamma p^2$ as we did for $H$ by itself, we should be correcting for 
\begin{align*}
-\gamma p^2 + c\gamma Tp^2=  - \gamma (1-cT) p^2 
\end{align*}           
in equation~\eqref{eqn:expVqv}.  Note that such a correction is possible when $c<1/T$ as the term above is negative, thus dissipative.  In fact, by definition of $\Psi$ we should replace $V_1$ above by 
\begin{align*}
V_\delta (q,p) = \exp( c(H(q,p) + \delta \Psi))
\end{align*} 
where $\delta = 1- cT$.  Following the same line of reasoning as above and again assuming $\Psi \in C^2$ for simplicity, we can then arrive at the desired bound whenever $c<1/T$.  

The next result summarizes this observation without of course making the false assumption that $\Psi \in C^2$.   

\begin{theorem}
\label{thm:LyapPr}
Fix $c\in (0, 1/T)$ and define $\delta = 1-cT$.  Then there exists $\xi_*>0$ and $h_0(\xi_*)>0$ large enough such that the It\^{o} differential of $V_\delta= \exp(c(H+ \delta \Psi))$ exists and satisfies 
\begin{align*}
d V_\delta(q_t, p_t) \leq [- C_1 V_\delta(q_t, p_t) + C_2] \,dt +  \Sigma_\delta(q_t, p_t) \, dW_t
\end{align*}
for some positive constants $C_1, C_2$ and some locally bounded, measurable mapping $\Sigma_\delta : \H \rightarrow \R$.  
\end{theorem}

 Lastly, we state and prove the following lemma which together with Theorem~\ref{thm:LyapPr} implies Theorem~\ref{MainThm}.  Its proof follows a now standard path \cite{MattinglyStuartHigham02}. 

\begin{lemma}\label{smallset}
  For every $\eta >0$, there exists a probability measure $\nu$
  supported in $\H$, a $t >0$  and $c_0 >0$ so that for all $A\subset \H$ Borel
  \begin{align*}
\inf_{\{(q,p)\in \H \,:\, H(q,p) \leq \eta\}}    \mathcal{P}_t((q,p), A) \geq c_0 \nu(A).  \end{align*}
\end{lemma}
\begin{proof}
  Let $\mathcal{L}$ denote the generator of the Markov semigroup
  associated to \eqref{sde1} and $\mathcal{L}^*$ denote the formal adjoint of $\mathcal{L}$ with respect to the $L^2$ inner product.  We begin by observing that the operators $\partial_s \pm \mathcal{L}$, $\partial_s \pm \mathcal{L}^*$, $\mathcal{L}$, $\mathcal{L}^*$ are hypoelliptic (see \cite{MattinglyStuartHigham02} for the
  straightforward calculation of Lie-brackets).  For every $x_0=(q_0, p_0) \in \H$, this implies that the transition measure
  $\mathcal{P}_s(x_0,\, \cdot \,)$ possesses a probability density function $\rho_s(x_0, y)$ (with respect to Lebesgue measure $dy$ on $\H$) which is a $C^\infty$ function on $(0, \infty)\times \H \times \H$.  In particular, we may write 
  \begin{align*}
    \mathcal{P}_s(x_0,B_\delta(x_0)) = \int_{B_\delta(x_0)} \rho_s(x_0,y)dy
  \end{align*}
  for $x_0\in \H$, $s>0$, and a sufficiently small $\delta$-ball around $x_0$.  Since
  for small enough $s$, 
  $\mathcal{P}_s(x_0,B_\delta(x_0))>0$ there exists $y_0 \in
  B_\delta(x_0)$, a $c'_0 >0$ and a possibly smaller $\delta$ such that
\begin{align*}
  \inf_{ (x,y) \in B_\delta(x_0)  \times B_\delta(y_0) } \rho_s(x,y)
  \geq c_0' >0\,
\end{align*}
as the function $(x,y)\mapsto\rho_s(x,y)$ is continuous.  

Now one can follow Lemma~3.4 of \cite{MattinglyStuartHigham02} to
construct a control argument ensuring that given any open set
$\mathcal{O} \subset \H$ and $\H(\eta)=\{ (q,p) : H(q,p) \leq \eta\}$,
there exists a $t >0$ and $c_0''>0$ such that
\begin{align*}
  \inf_{z \in \H(\eta) } \mathcal{P}_t(z,\mathcal{O}) \geq c_0''  .  
\end{align*}
The argument in \cite{MattinglyStuartHigham02} assumes that the drift
vector field is bounded on compact sets. This is still true if we
restrict to $\H(\eta)$ for any finite $\eta >0$. The uniform lower
bound is not explicitly mentioned, however one can pick a single
tubular neighborhood size of the needed control and ensure that the
control and its derivatives are uniformly bounded for all starting and
ending points in $\H(\eta)$.

Setting $t=r+s$, defining $\nu$ as normalized Lebesgue measure on
$B_\delta(y_0)$ and combining the preceding two estimates produces,
for any $z \in \H(\eta)$ and $A \subset \H$
\begin{align*}
  \mathcal{P}_t(z, A) =& \int_{\H} \mathcal{P}_r(z,dy)\mathcal{P}_s(y,A) \\
  \geq& \int_{B_\delta(x_0)} \mathcal{P}_r(z,dy)\mathcal{P}_s(y,A\cap B_\delta(y_0)) \\
\geq& \int_{A\cap B_\delta(y_0)} \int_{B_\delta(x_0)} \mathcal{P}_r(z,dy)\mathcal{P}_s(y
, dz) \\
\geq& c_0''\lambda_{leb}(A\cap B_\delta(y_0)) \int_{B_{\delta(x_0)}}
\mathcal{P}_r(z,dy)\\
\geq & c_0' c_0'' \lambda_{leb}(A\cap B_\delta(y_0)) =c_0'
  c_0'' \lambda_{leb}(B_\delta(y_0)) \nu(A) 
\end{align*}
which concludes the proof.
\end{proof}

\begin{proof}[Proof of Theorem~\ref{MainThm}]
Theorem~\ref{MainThm} follows by combining Theorem~\ref{thm:LyapPr} and
Lemma~\ref{smallset} and invoking Theorem~1.2 from
\cite{HairerMattingly08}. This result is a repackaging of a well known
result of Harris. It can be found in many places. Most appropriate for
the current discussion is the work of Meyn and Tweedie exemplified by
\cite[Section 15]{MeynTweedie93}.
\end{proof}

\section{Proof of Theorem~\ref{thm:Lyap} and Theorem~\ref{thm:LyapPr}}
\label{sec:lyap_proof}

To help setup the statement of the lemma, which will be used to prove both results, define the boundary functions on $\{q\in \R\,: \, q>0\}$ by $$c_0(q)=0, \, c_1(q)= -\xi_* q^{-\frac\beta2}, \, c_2(q)= -\xi_* q^{\frac{\alpha}2},\, c_3(q)=\xi_* q^{\frac{\alpha}2}, \, c_4(q)= \xi_* q^{-\frac\beta2},$$ and let $l^i_t$ denote the local time of the process $(q_t, p_t)$ on the curve $p=c_i(q)$, $q>0$, on the time interval $[0,t]$ given by 
\begin{align*}
l^i_t \eqdef \lim_{\epsilon \downarrow 0} \frac{1}{2\epsilon} \int_0^t \mathbf{1}\{-\epsilon < p_s- c_i(q_s)< \epsilon \} d \langle\langle p- c_i(q), p-c_i(q)\rangle \rangle_s 
\end{align*} 
where the limit above is in probability.  We recall that the corrector $\Psi: \H \rightarrow \R$ was defined to be $C^2$ except possibly on the collection of nonintersecting curves $$C_i\eqdef \{ (q,p) \in \H \, : \, p=c_i(q), \, H(q,p) \geq h_0 \}.$$  Therefore for any function $\Phi:  \H\rightarrow \R$ which is continuous except possibly on $\bigcup_{i=0}^4 C_i$, whenever the following quantities exist  we let
\begin{align*}
\Phi(q,p^\pm) &=  \lim_{(Q,P) \rightarrow (q,p)}  \Phi(Q,P) \,\,  \text{ if } \Phi \text{ is continuous at } (q,p) \\
\Phi(q,p^+) &= \lim_{\substack{(Q,P) \rightarrow (q,p)\\P>c_i(Q)}}\Phi(Q,P)  \,\, \text{ if } (q,p) \in C_i\\
\Phi(q,p^-) &= \lim_{\substack{(Q,P) \rightarrow (q,p)\\ P<c_i(Q)}}\Phi(Q,P) \,\, \text{ if } (q,p) \in C_i .  
\end{align*}  

\begin{lemma}
\label{lem:Lyappsi}
Let $h(t) \eqdef H(q_t, p_t)$.  Then the It\^{o} differential of $\Psi(q_t, p_t)$ exists and satisfies
\begin{align}
\label{eqn:peskir}
d\Psi(q_t, p_t) &= \frac{1}{2} (\mathcal{L}\Psi)(q_t, p_t^+) \, dt + \frac{1}{2} (\mathcal{L} \Psi)(q_t, p_t^-) \, dt \\
\nonumber &\,\, + \frac{\sigma}{2} \partial_p \Psi(q_t, p_t^+) \, dW_t + \frac{\sigma}{2} \partial_p \Psi(q_t, p_t^-) dW_t \\
\nonumber &\,\,  + \frac{1}{2}\sum_{i=0}^4 (\partial_p\Psi(q_t, p_t^+) - \partial_p\Psi(q_t, p_t^-) ) \mathbf{1}\{ p_t = c_i(q_t), h(t) \geq h_0\} \, dl^i_t .  
\end{align}
Moreover for each $\epsilon >0$, we can choose the parameters $c_i^\pm>0$ such that for all $\xi_*>0$ large enough there exists $h_0=h_0(\xi_*)>0$ large enough so that 
\begin{enumerate}
\item[(a)]  The local time contribution is nonpositive, i.e., $$ \frac{1}{2}\sum_{i=0}^4 (\partial_p\Psi(q, p^+) - \partial_p\Psi(q, p^-) ) \mathbf{1}\{ p= c_i(q), H(q,p) \geq h_0\} \leq 0.$$
\item[(b)] $\Psi(q,p) = o(H(q,p))$ and 
\begin{align*}
\frac{1}{2}(\mathcal{L}\Psi)(q, p^+) + \frac{1}{2} (\mathcal{L} \Psi)(q, p^-) \leq \gamma  p^2  - \gamma (\Lambda_* - \epsilon) H(q,p)  + o(H(q,p)).  
\end{align*} 
as $H(q,p)\rightarrow \infty$. 
\item[(c)]  There exist constants $C, D>0$ such that 
\begin{align*}
|\partial_p \Psi(q, p^+)| + | \partial_p \Psi(q,p^+)| \leq C H(q,p)^{\alpha^{-1}} + D
\end{align*}  
for all $(q,p) \in \H$.  
\end{enumerate}
\end{lemma}

Taking $V=H+ \Psi$ and $V_\delta= \exp(c(V+ \delta \Psi))$ where $c\in (0,1/T)$ is fixed and $\delta =1 -cT$, it is not hard to show that Lemma~\ref{lem:Lyappsi} along with Peskir's formula~\cite{Peskir_07} implies Theorem~\ref{thm:Lyap} and Theorem~\ref{thm:LyapPr}.

To prove Lemma~\ref{lem:Lyappsi}, we need the following definition.

\begin{definition}
\label{def:asym}
Let $X$ be a subset of $\H$ which possibly depends on $\xi_*$ having the property that for every $\xi_*>0$ there exists a sequence of points $\{(q_n, p_n)\} \subset X$ satisfying $H(q_n, p_n)\rightarrow \infty$ as $n\rightarrow \infty$.  For two functions $f, g: X\rightarrow \R\setminus\{0\}$, perhaps depending on $\xi_*$, we write $f\sim_X g$ if for every $\epsilon >0$ there exists $\xi_*>0$ and $h=h(\xi_*)>0$ large enough such that for all $(q,p) \in X$ with $H(q,p) \geq h$ we have 
\begin{align*}
1-\epsilon  \leq \frac{f(q,p)}{g(q,p)} \leq 1+\epsilon .  
\end{align*}
Also, for functions $f,g:X \rightarrow (0, \infty)$, possibly depending on $\xi_*$, we write $f \precsim_X g$ if for every $\epsilon >0$ there exists $\xi_*>0$ and $h=h(\xi_*)>0$ large enough such that for all $(q,p) \in X$ with $H(q,p) \geq h$ we have 
\begin{align*}
\frac{f(q,p)}{g(p,q)} \leq 1+\epsilon.      
\end{align*}
\end{definition}

\begin{remark}
This notation will be used heavily in the rest of the paper.  It is convenient in that it simplifies the asymptotic expressions that follow, as it allows us to see what happens first when $\xi_*>0$ is chosen large and then, subsequently, when the energy parameter $h$ is taken to infinity in various regions of $\H$.   
\end{remark}

\begin{proof}[Proof of Lemma~\ref{lem:Lyappsi}]
The fact that $\Psi$ has an It\^{o} differential and that it satisfies the formula~\eqref{eqn:peskir} follows from Peskir's formula~\cite{Peskir_07}, the boxed formulas in the Appendix and the fact that $j^{-1}, k^{-1}$ introduced above~\eqref{eqn:jk1} are $C^2$.  The boxed formulas in the Appendix show that $f_i, g_i\in C^2(\mathcal{S}_i(\xi_*, h_0):[0, \infty))$ and that $F_i, G_i \in C^2([h_0, \infty): [0, \infty)).$  The remaining regularity requirements needed to apply Peskir's formula follow immediately by the boundary conditions satisfied by the $\Psi_i^\pm$'s and since
\begin{align*}
\frac{1}{2}\overline{\Psi}_1^+(h) + \overline{\Psi}_2^+(h) + \frac{1}{2}\overline{\Psi}^+_3(h) = \frac{1}{2}\overline{\Psi}_1^-(h) + \overline{\Psi}_2^-(h) + \frac{1}{2}\overline{\Psi}^-_3(h)   =0  
\end{align*}   
for all $h\geq h_0$.   
 
We now turn to establishing conclusions (a),(b), and (c) of the result.  Let $\epsilon >0$ be small.  We first establish conclusion (a) concerning the sign of the local time contribution in formula~\eqref{eqn:peskir}.  In total, there are six calculations that need to be performed: two on the positive $p$ side of $\H$ on the boundaries $$\mathcal{S}_{12}^+\eqdef\mathcal{S}_1^+ \cap \mathcal{S}_{2}^+\, \text{ and } \, \mathcal{S}_{23}^+\eqdef \mathcal{S}_{2}^+ \cap \mathcal{S}_3^+ ,$$ two on the negative $p$ side of $\H$ on the boundaries $$\mathcal{S}_{12}^- \eqdef \mathcal{S}_1^- \cap \mathcal{S}_2^- \, \text{ and } \, \mathcal{S}_{23}^-\eqdef\mathcal{S}_{2}^- \cap \mathcal{S}_3^- ,$$ and two where $p=0$ on the boundaries $$\mathcal{S}_{10}\eqdef \mathcal{S}_1 \cap \{(q,p) \in \H\,: p=0 \} \,\, \text{ and } \, \, \mathcal{S}_{30}\eqdef \mathcal{S}_3 \cap \{ (q,p) \, : \, p=0\}.$$  We start with the boundary calculations on the positive $p$ side of $\H$, beginning with $\mathcal{S}_{12}^+$.

Observe that for $(q,p) \in \mathcal{S}_{12}^+ $ and $h=H(q,p)$
\begin{align*}
&\partial_p \Psi(q, p^+) -\partial_p \Psi(q, p^-)\\
&= \partial_p(\Psi_2^+(q,p) + \tfrac{1}{2}\overline{\Psi}_3^+(h)) - \partial_p(\Psi_1^+(q,p)+ \overline{\Psi}_2^+(h)+ \tfrac{1}{2} \overline{\Psi}_3^+(h))\\
&= \partial_p (\Psi_2^+(q,p) - \overline{\Psi}_2^+(h)) - \partial_p(\Psi_1^+(q,p))\\
&= c_2^+ \gamma[ \mathcal{A}_+(h) \partial_p(g_2-\tfrac{1}{2} G_2) -\partial_p( f_2- \tfrac{1}{2}F_2)] - c_1^+ \gamma [\mathcal{A}_+(h) \partial_pg_1- \partial_p f_1].
\end{align*}
Applying formulas~\eqref{eqn:FS1b}, \eqref{eqn:FS2b} and \eqref{eqn:AF} in the Appendix, we find that provided $c_1^+ \neq c_1^+$
\begin{align*}
\partial_p \Psi(q, p^+) -\partial_p \Psi(q, p^-)\sim_{\mathcal{S}_{12}^+} (c_2^+-c_1^+) \gamma \times \text{positive quantity}
\end{align*}
as $2\alpha/(\alpha+2) <2$ for $\alpha >2$.  By picking $c_2^+ < c_1^+$, we find that for all $\xi_*>0$ and all $h_0(\xi_*)>0$ large enough, the local time contribution from this boundary is nonpositive.

We now move on to the next and last boundary $\mathcal{S}_{23}^+$ on the positive $p$ side of $\H$.  Note that for $(q,p) \in \mathcal{S}_{23}^+$ and $h=H(q,p)$
\begin{align*}
&\partial_p \Psi(q, p^+) -\partial_p \Psi(q, p^-)
\\
&= \partial_p(\Psi_2^+(q, p) +\tfrac{1}{2}\overline{\Psi}_3^+(h))  - \partial_p(\Psi_3^+(q,p)- \tfrac{1}{2}\overline{\Psi}_3^+(h))\\
&=\partial_p(\Psi_2^+(q,p)) -\partial_p(\Psi_3^+(q,p)-\overline{\Psi}_3^+(h)) \\
&= c_2^+\gamma[\mathcal{A}_+(h) \partial_p g_2 - \partial_{p} f_2 ] - c_3^+\gamma [ \mathcal{A}_+(h) \partial_p(g_3-G_3)-\partial_p(f_3-F_3)]\\
&\qquad + c_3^+ \gamma\partial_p(\mathcal{A}_{+}(h)) G_3(h)( (j'(h))^{-1}-1).    
\end{align*}
Applying formulas~\eqref{eqn:FS2b}, \eqref{eqn:FS3b} and ~\eqref{eqn:AF} in the Appendix as well as~\eqref{eqn:jk2}, we find that so long as $c_2^+ \neq c_3^+$
\begin{align*}
\partial_p \Psi(q, p^+) -\partial_p \Psi(q, p^-)\sim_{\mathcal{S}_{23}^+} (c_2^+ - c_3^+)\gamma \times \text{negative quantity} 
\end{align*}
as $2\alpha/(\alpha+2) < 2$.  By picking $c_3^+< c_2^+$, we find that for all $\xi_*>0$ large enough and all $h_0(\xi_*)>0$ large enough, the local time contribution from this boundary is nonpositive.

We now perform the boundary-flux calculations on the negative $p$ side of $\H$, starting with $\mathcal{S}_{12}^-$.  Note that for $(q,p) \in \mathcal{S}_{12}^-$ and $h=H(q,p)$:
\begin{align*}
&\partial_p \Psi(q, p^+) -\partial_p \Psi(q, p^-)\\
&= \partial_p(\Psi_1^-(q, p) -\tfrac{1}{2} \overline{\Psi}_1^-(h))   - \partial_p(\Psi_2^-(q,p)+ \tfrac{1}{2}\overline{\Psi}_1^-(h))\\
&=\partial_p(\Psi_1^-(q, p) - \overline{\Psi}_1^-(h))   - \partial_p(\Psi_2^-(q,p))\\
&= c_1^- \gamma[\mathcal{A}_-(h) \partial_p (g_1-G_1) -  \partial_p(f_1-F_1)] - c_2^- \gamma [\mathcal{A}_-(h) \partial_p g_2- \partial_p f_2]\\
&\qquad + c_1^- \gamma \partial_p( \mathcal{A}_{-}(h)) G_1(h) ((k'(h))^{-1}-1).    
\end{align*}
Applying formulas~\eqref{eqn:FS3b}, \eqref{eqn:FS2b} and~\eqref{eqn:AF} in the Appendix as well as~\eqref{eqn:jk2}, we find that so long as $c_1^- \neq c_2^-$
\begin{align*}
\partial_p \Psi(q, p^+) -\partial_p \Psi(q, p^-)\sim_{\mathcal{S}_{12}^-} (c_1^- - c_2^-) \gamma \times \text{positive quantity}
\end{align*}
again since $2\alpha/(\alpha+2) < 2$.  By picking $c_1^-< c_2^-$, we find that for all $\xi_*>0$ and $h_0(\xi_*)>0$ large enough, the local time contribution from this boundary is nonpositive.

We now move onto the last boundary $\mathcal{S}_{23}^-$ on the negative $p$ side of $\H$.   
Note that for $(q,p) \in \mathcal{S}_{23}^-$ and $h=H(q,p)$:
\begin{align*}
&\partial_p \Psi(q, p^+) -\partial_p \Psi(q, p^-)\\
&= \partial_p(\Psi_3^-(q, p) +\tfrac{1}{2} \overline{\Psi}_1^-(h) + \overline{\Psi}^-_2(h))   - \partial_p(\Psi_2^-(q,p)+ \tfrac{1}{2}\overline{\Psi}_1^-(h))\\
&=\partial_p(\Psi_3^-(q, p))   - \partial_p(\Psi_2^-(q,p)- \overline{\Psi}^-_2(h))\\
&= c_3^- \gamma[\mathcal{A}_-(h) \partial_p g_3 -  \partial_pf_3] - c_2^- \gamma [\mathcal{A}_-(h) \partial_p( g_2- \tfrac{1}{2}G_2)- \partial_p (f_2-\tfrac{1}{2} F_2)]
\end{align*}
Applying formulas~\eqref{eqn:FS3b}, \eqref{eqn:FS2b} and~\eqref{eqn:AF} in the Appendix, we find that so long as $c_1^- \neq c_2^-$
\begin{align*}
\partial_p \Psi(q, p^+) -\partial_p \Psi(q, p^-)\sim_{\mathcal{S}_{21}^-} (c_3^- - c_2^-) \gamma \times \text{negative quantity}
\end{align*}
again since $2\alpha/(\alpha+2) < 2$.  By picking $c_2^-< c_3^-$, we find that for all $\xi_*>0$ and $h_0(\xi_*)>0$ large enough, the local time contribution from this boundary is nonpositive.  

Thus far, the parameters $c_i^\pm>0$ have been picked to satisfy
\begin{align*}
c_1^+ > c_2^+ >c_3^+\,\,, \text{ and } c_1^-< c_2^- < c_3^-. 
\end{align*}
On the final two boundaries $\mathcal{S}_{10}$ and $\mathcal{S}_{30}$ we must make sure these choices can be respected.  We begin with the boundary $\mathcal{S}_{10}$.  Since 
\begin{align*}
\tfrac{1}{2}\overline{\Psi}_1^+(h) + \overline{\Psi}_2^+(h) + \tfrac{1}{2}\overline{\Psi}_3^+(h)=0
\end{align*}
and $\partial_pf_1=0$ on $\mathcal{S}_{10}$ (see formulas~\eqref{eqn:FS1b}), we find that for $(q,p)\in \mathcal{S}_{10}$ and $h=H(q,p)$: 
\begin{align*}
\partial_p \Psi(q,p^+)- \partial_p \Psi(q,p^-) &= \partial_p (\Psi_1^+(q,p) + \overline{\Psi}_2^+(h) + \tfrac{1}{2}\overline{\Psi}_3^+(h)) - \partial_p( \Psi_1^-(q,p) - \overline{\Psi}_1^-(h))\\
&= \partial_p ( \Psi_1^+(q,p) - \tfrac{1}{2}\overline{\Psi}_1^+(h)) - \partial_p( \Psi_1^-(q,p) - \overline{\Psi}_1^-(h))\\
&= \partial_p (\Psi_1^+(q,p) ) -  \partial_p (\Psi_1^-(q,p))\\
&= c_1^+ \gamma \mathcal{A}_{+}(k^{-1}(K(q,p))) \partial_p g_1 - c_1^- \gamma \mathcal{A}_- (k^{-1}(K(q,p))) \partial_p g_1.
\end{align*}  
Thus since $\partial_p g_1 <0$ on $\mathcal{S}_{10}$ (see formulas~\eqref{eqn:FS1b}), so long as $c_1^+ \neq c_1^-$ we have that 
\begin{align*}
&\partial_p \Psi(q,p^+)- \partial_p \Psi(q,p^-)\sim_{\mathcal{S}_{10}} (c_1^+ - c_1^-) \gamma \times \text{negative quantity}.  
\end{align*}
Hence by picking $c_1^+ >c_1^-$, we find that for all $\xi_*>0$ and $h_0(\xi_*)>0$ large enough, the local time contribution from this boundary is nonpositive.  

We now move on to the last boundary $\mathcal{S}_{30}$.  Since 
\begin{align*}
\tfrac{1}{2}\overline{\Psi}_1^-(h) + \overline{\Psi}_2^-(h) + \tfrac{1}{2}\overline{\Psi}_3^-(h)=0
\end{align*}
and $\partial_pf_3 =0$ on $\mathcal{S}_{30}$ (see formulas~\eqref{eqn:FS3b}), 
we find that for $(q,p) \in \mathcal{S}_{30}$ and $h=H(q,p)$:
\begin{align*}
\partial_p \Psi(q,p^+)- \partial_p \Psi(q,p^-) &= \partial_p (\Psi_3^+(q,p) -\tfrac{1}{2} \overline{\Psi}_3^+(h)) - \partial_p( \Psi_3^-(q,p) - \tfrac{1}{2}\overline{\Psi}_3^-(h))\\
&= \partial_p (\Psi_3^+(q,p) ) -  \partial_p (\Psi_3^-(q,p))\\
&= c_3^+ \gamma \mathcal{A}_{+}(j^{-1}(J)) \partial_p g_3 - c_3^- \gamma \mathcal{A}_- (j^{-1}(J)) \partial_p g_3.
\end{align*}  
Thus since $\partial_p g_3 >0$ on $\mathcal{S}_{30}$ (see formulas~\eqref{eqn:FS3b}), so long as $c_3^+ \neq c_3^-$ we have that 
\begin{align*}
&\partial_p \Psi(q,p^+)- \partial_p \Psi(q,p^-)\sim_{\mathcal{S}_{30}} (c_3^+ - c_3^-) \gamma \times \text{positive quantity}.  
\end{align*}
Hence by picking $c_3^+ < c_3^-$, we find that for all $\xi_*>0$ and $h_0(\xi_*)>0$ large enough, the local time contribution from this boundary is nonpositive.

To summarize the proof so far in part (a), by picking
\begin{align}
\label{eqn:bfparam}
c_1^+ > c_2^+ >c_3^+, \,\,\,\, c_1^- < c_2^- < c_3^-, \,\,\,\, c_1^+>c_1^-, \,\, \,\,c_3^- > c_3^+
\end{align}
 for all $\xi_*>0$ and $h_0(\xi_*)>0$ large enough
 \begin{align*}
  \frac{1}{2}\sum_{i=1}^4 (\partial_p\Psi(q, p^+) - \partial_p\Psi(q, p^-) ) \mathbf{1}\{ p= c_i(q), H(q,p) \geq h_0\} \leq 0.  
   \end{align*} 
   But let us for a moment see how we can pick the parameters $c_i^\pm$ in this way with $c_i^\pm \leq 1$ and with $c_i^\pm$ close to $1$.  This will be important in part (b).  Recall $\epsilon >0$ small was fixed and define $c_1^+=1$, $c_2^+=1-\tfrac{1}{2}\epsilon$, $c_3^+=1-\tfrac{7}{8}\epsilon$, $c_1^- =1-\tfrac{7}{8}\epsilon$, $c_2^-=1-\tfrac{3}{4}\epsilon$, and $c_3^-=1- \tfrac{1}{2}\epsilon.$  Then notice that the relationships~\eqref{eqn:bfparam} are respected and that $c_i^\pm \leq 1$.         
   
   We now work on establishing part (b) of the result.  The fact that $\Psi(q,p)=o(H(q,p))$ as $H(q,p)\rightarrow \infty$ follows easily by combining the formulas~\eqref{eqn:FS1a}, \eqref{eqn:FS2a}, \eqref{eqn:FS3a} and \eqref{eqn:AF} with $i=0$ in the Appendix with the formulas \eqref{eqn:jk1} to produce an asymptotic bound for $\Psi$ in each region.  More precisely, we find that for $h=H(q,p)$ 
   \begin{align*}
   |\Psi(q,p)| \precsim_{\H}  C(\xi_*) h^{\frac{1}{2}+\alpha^{-1}} 
   \end{align*}  
   for some positive constant $C(\xi_*)$.  Since $\alpha >2$, this finishes the proof that $\Psi(q,p)= o(H(q,p)) $ as $H(q,p) \rightarrow \infty$.  Now we check the claimed bound on the generator applied to $\Psi$.  This will be done region by region.


We begin in the region $\mathcal{S}_1^+$  and obtain the necessary estimate for 
\begin{align*}
\mathcal{L}(\Psi_1^+(q,p) +  \overline{\Psi}_2^+(h) + \tfrac{1}{2}\overline{\Psi}_3^+(h))=\mathcal{L} (\Psi_1^+(q,p) - \tfrac{1}{2} \overline{\Psi}_1^+(h))
\end{align*}
where $(q,p) \in \mathcal{S}_1^+$ and $h=H(q,p)$.  Observe that for $(q,p) \in \mathcal{S}_1^+$ and $h=H(q,p)$ we may write:  
\begin{align*}
&\mathcal{L} ( \Psi_1^+(q,p) - \tfrac{1}{2} \overline{\Psi}_1^+(h))\\
&= \mathcal{K}\Psi_1^+(q,p)+ (\mathcal{L}-\mathcal{K})(\Psi_1^+(q,p))-\tfrac{1}{2}(\mathcal{L}- \mathcal{H})(\overline{\Psi}_1^+(h)) \\
&=  -\gamma[\mathcal{A}_{+}(k^{-1}(K(q,p)))- p^2]  +  (\mathcal{L}-\mathcal{K})(\Psi_1^+(q,p))-\tfrac{1}{2}(\mathcal{L}- \mathcal{H})(\overline{\Psi}_1^+(h))
\end{align*} 
since $\mathcal{H} (\overline{\Psi}_1^+(h))=0$.  Note that in the last equality we used the fact that $c_1^+ = 1$.  Note also that for $h=H(q,p)$: 
\begin{align*}
|-U'(q)-b \beta q^{-\beta-1}-\gamma p| \precsim_{\mathcal{S}_1^+} C\max\{ o(h^{1+\beta^{-1}}), h^{\frac{1}{2}}\}.  
\end{align*}
Therefore, applying formulas~\eqref{eqn:FS1a}, \eqref{eqn:AF} and \eqref{eqn:jk1}-\eqref{eqn:jk3} produces the required formula
\begin{align*}
\mathcal{L}(\Psi_1^+(q,p) +  \overline{\Psi}_2^+(h) + \tfrac{1}{2}\overline{\Psi}_3^+(h)) \leq -\gamma( \Lambda_*- \epsilon)  h + \gamma p^2 + o(h) .  
\end{align*}


Moving onto region $\mathcal{S}_2^+$, notice that for $(q,p) \in \mathcal{S}_2^+ $ and $h=H(q,p)$ we have: 
\begin{align*}
\mathcal{L}(\Psi^+_2(q,p) + \tfrac{1}{2} \overline{\Psi}_3^+(h)  )&=  \mathcal{H} \Psi_2^+(q,p) + (\mathcal{L}-\mathcal{H}) (\Psi_2^+(q,p) + \tfrac{1}{2} \overline{\Psi}_3^+(h)) \\
&= - (1-\tfrac{1}{2}\epsilon) \gamma [\mathcal{A}_{+}(h)-p^2] + (\mathcal{L}-\mathcal{H}) (\Psi_2^+(q,p)+\tfrac{1}{2} \overline{\Psi}_3^+(h) ) 
\end{align*} 
where in the last equality we used the fact that $c_2^+=1-\tfrac{1}{2}\epsilon.$  Applying the formulas~\eqref{eqn:FS2a}, \eqref{eqn:AF} and \eqref{eqn:jk1}-\eqref{eqn:jk3}  and the fact that $\alpha >2$ we obtain the required bound in $\mathcal{S}_2^+$
\begin{align*}
\mathcal{L}(\Psi^+_2(q,p) + \tfrac{1}{2} \overline{\Psi}_3^+(h)  )\leq - \gamma( \Lambda_* -\epsilon) h + \gamma p^2 + o(h).  
\end{align*}


In the region $\mathcal{S}_3^+$, notice that for $(q,p) \in \mathcal{S}_3^+ $ and $h=H(q,p)$ we have 
\begin{align*}
&\mathcal{L}(\Psi^+_3(q,p) - \tfrac{1}{2} \overline{\Psi}_3^+(h)  )\\
&= \mathcal{J} \Psi_3^+(q,p) + (\mathcal{L}-\mathcal{J}) (\Psi_3^+(q,p) ) -\tfrac{1}{2}(\mathcal{L}-\mathcal{H})(\overline{\Psi}_3^+(h)) \\
&= - (1- \tfrac{7}{8}\epsilon) \gamma [\mathcal{A}_+(j^{-1}(J))- p^2] + (\mathcal{L}-\mathcal{J}) (\Psi_3^+(q,p) ) -\tfrac{1}{2}(\mathcal{L}-\mathcal{H})(\overline{\Psi}_3^+(h))
\end{align*} 
where in the last equality we used the fact that $c_3^+=1-\tfrac{7}{8}\epsilon$.  To help estimate the remainder term $(\mathcal{L}-\mathcal{J}) (\Psi_3^+(q,p)),$ first note that for $h=H(q,p)$:
\begin{align*}
|-U'(q)+ a\alpha q^{\alpha-1}- \gamma p| \precsim_{\mathcal{S}_3^+} C \max\{o(h^{1-\alpha^{-1}}), h^{\frac{1}{2}} \}.  
\end{align*}
Therefore for $h=H(q,p)$ applying the formulas~\eqref{eqn:FS3a}, \eqref{eqn:AF} and \eqref{eqn:jk1}-\eqref{eqn:jk3} produces the necessary bound on $\mathcal{S}_3^+$:
\begin{align*}
\mathcal{L}(\Psi^+_3(q,p) - \tfrac{1}{2} \overline{\Psi}_3^+(h)  )\leq -\gamma(\Lambda_*-\epsilon) h + \gamma p^2 + o(h).  
\end{align*}


The arguments establishing the needed bounds in the regions $\mathcal{S}_i^-$, $i=1,2,3$, are done in a nearly identical fashion so we omit those details for brevity.  This finishes the proof of part (b).

The proof of part (c) is a straightforward consequence of the formulas on the first three pages of the Appendix and the formulas \eqref{eqn:jk1}-\eqref{eqn:jk3}.                 
\end{proof}

\section{Conclusion}
\label{sec:conclusion}

We began by observing that to leading order, the dynamics at high
energy follows the deterministic dynamics given by a modified
Hamiltonian perturbed by a small noise. To leverage this observation,
stochastic averaging techniques, built on auxiliary Poisson equation
methods, were used to construct a Lyapunov function sufficient to
prove exponential convergence to equilibrium.  The central result
given in Theorem~\ref{MainThm} covers important singular potentials,
including Lennard-Jones type potentials, which had not been covered by
previous results.  Theorem~\ref{MainThm} has two principal
remaining deficiencies. First it only applies to two interacting particles in
isolation.  Second, Theorem~\ref{MainThm} does not cover the classical
case where the confining potential $U$ grows quadratically at
infinity.

In principle, the extension to many particles could follow a similar
route, since when two particles are near each other their principal
interaction is with each other while other particles are just a small
perturbation. However it is possible that the orbit over which one
must average could also interact with other particles. This would make
finding closed form representations of the averaging measure difficult
at best (chaotic orbits are to be expected). Even if in some setting
the high energy orbits remain of the type considered here, the
combinatorics of the possible interactions would be complicated.

In contrast, the extension to potentials with quadratic growth is
almost certainly within reach. In fact, Figure ~\ref{fig:numerics}
gives a strong indication how to proceed. Since for $\alpha_1=2$ the
period of oscillation is not going to zero as the energy of the system
increases, instantaneous homogenization/averaging of the effect of one
orbit is not feasible. However, building on an idea from
\cite{Rey-BelletThomas:2002} one could consider the average of the
energy over one period $\tau$ of the system. First observe that $\tau$
has a limit $\tau_*>0$ as the energy goes to $\infty$.  Namely one
would consider the quantity
\begin{align*}
  V_t=\frac{1}{\tau_*}\int_{t}^{t+\tau_*} \E[ H(q_s,p_s)\,|\,\mathcal{F}_t] ds\,.
\end{align*}
Then using \eqref{Heq} one obtains
\begin{align*}
  \frac{\partial V_t}{\partial t}&= \frac{1}{\tau_*}
  \big(\E [ H(q_{t+\tau_*},p_{t+\tau_*})\,|\,\mathcal{F}_t]-H(q_{t},p_{t}) \big)\\
  &= -
  \frac{\gamma}{\tau_*} \int_{t}^{t+\tau_*} \E[ p_s^2 \,|\,\mathcal{F}_t]ds + \frac{\sigma^2}2\,. 
\end{align*}
Since at high energy $p_s$ will be very close to the deterministic
orbit, one can likely prove that
\begin{align*}
 \frac{1}{\tau_*}\E  \int_{t}^{t+\tau_*} p_s^2 ds\approx \frac{1}{\tau_*}
 \int_{t}^{t+\tau_*} \E\big[ \rate(H(q_s.p_s))H(q_s,p_s)\big]ds\approx \rate_*  V_t \,.
\end{align*}
This could then be used to obtain control of the excursions away from
the center of space. Note that following the above argument will not
produce an Lyapunov function which is infinitesimally decreasing on
average as was constructed in this paper. This argument essentially
amortizes the total energy dissipation that occurs over a single
orbit, smoothing out the times when the infinitesimal rate of energy
dissipation nears zero. We felt that covering the quadratic case
is not sufficient motivation for the extra complications.

\vspace{1em}
\noindent\textbf{Acknowledgments.}
The authors wishes to thank the NSF for its support through grants
DMS-0449910 (JCM), DMS-0854879 (JCM), DMS-1613337 (JCM), DMS-0204690 (SCS), and DMS-1612898 (DPH). JCM would like to thank Martin Hairer and Luc Rey-Bellet for interesting and
informative discussions. We would also like to thank a careful referee
for finding a mistake in an earlier version of the paper as well as for making several helpful comments and suggestions.  

\bibliography{refs}
 \bibliographystyle{alpha}
 
\section*{Appendix}

In what follows, $C(\xi_*)$ will denote a generic positive constant
depending on $\xi_*$.  Also, below we write $\partial_p^0\eqdef1$, $f_i=f_i(Q,P)$, $g_i=g_i(Q,P)$, and $F_i=F_j(h)$, $G_i=G_i(h)$, $\mathcal{A}_\pm=\mathcal{A}_{\pm}(h)$ where $h=H(Q,P)$.  Recalling $\H_{h_0}= \bigcup_{i} \mathcal{S}_i$ and the notation
$\sim_{X}$ and $\precsim_X$ introduced in Definition~\ref{def:asym},
here we will establish the following formulas ($i=0,1,2$ below). For notational compactness, we also introduce
boundary sets
$\mathcal{S}_{k\ell}^\pm= \mathcal{S}_k^\pm \cap \mathcal{S}_\ell^\pm$, $\mathcal{S}_{10}\eqdef \mathcal{S}_1 \cap \{(Q,P) \in \H \, : \, P=0\}$ and $\mathcal{S}_{30}\eqdef \mathcal{S}_3 \cap \{(Q,P) \in \H \, : \,  P=0\}.$ 
\begin{equation}
\label{eqn:FS1a}
\tag{FS1a}
\begin{array}{lr}
|\partial_P^i f_1| \precsim_{\mathcal{S}_1}C(\xi_*) h^{\frac{1}{2}- \frac{1}{2}i-\betaI},&\,\, |\partial_P^i g_1|\precsim_{\mathcal{S}_1}C(\xi_*) h^{-\frac{1}{2}-\frac{1}{2}i-\betaI}\\\\
|\partial_P^i F_1| \precsim_{\H_{h_0}}C(\xi_*) h^{\frac{1}{2}- \frac{1}{2}i-\betaI}& \,\,  
|\partial_P^i G_1| \precsim_{\H_{h_0}} C(\xi_*) h^{-\frac{1}{2}-\frac{1}{2}i-\betaI}
\end{array}
\end{equation}

\begin{align}
\nonumber &\partial_P(f_1- F_1) \sim_{\mathcal{S}_{12}^-}, \, \partial_P f_1 \sim_{\mathcal{S}_{12}^+}  - 2^{1-\betaI}\betaI \xi_*^{\frac2\beta} h^{-\betaI}\\ \nonumber \\ 
\nonumber &\partial_P(g_1-G_1) \sim_{\mathcal{S}_{12}^-}, \, \partial_P g_1 \sim_{\mathcal{S}_{12}^+}- 2^{-\betaI} \betaI \xi_*^{\frac2\beta} h^{-1-\betaI}\\ \label{eqn:FS1b}
\tag{FS1b}\\ \nonumber 
&\partial_P g_1 <0 \text{ on } \mathcal{S}_{10}, \,\,  \partial_P f_1=0 \text{ on } \mathcal{S}_{10} 
\\ \nonumber \\ \nonumber 
&\nonumber F_1 \sim_{\H_{h_0}} 2^{\frac{3}{2}-\betaI} \xi_*^{\frac2\beta} h^{\frac{1}{2}-\betaI}, \qquad G_1   \sim_{\H_{h_0}} 2^{\frac{1}{2}-\betaI} \xi_*^{\frac2\beta} h^{-\frac{1}{2}-\betaI}
\end{align}

\begin{equation}
\label{eqn:FS3a}
\tag{FS3a}
\begin{array}{lr}
|\partial_P^{i} f_3| \precsim_{\mathcal{S}_3}C(\xi_*) h^{\frac{1}{2}-\frac{1}{2}i+\alphaI},&\,\, |\partial_P^i g_3|\precsim_{\mathcal{S}_3}C(\xi_*) h^{-\frac{1}{2}-\frac{1}{2}i+\alphaI}\\\\
|\partial_P^i F_3| \precsim_{\H_{h_0}} C(\xi_*) h^{\frac{1}{2}-\frac{1}{2}i+\alphaI}, &\,\, |\partial_P^i G_3 |\precsim_{\H_{h_0}} C(\xi_*) h^{-\frac{1}{2}-\frac{1}{2}i+\alphaI}
\end{array}
\end{equation}

\begin{align}
\nonumber \\ 
\nonumber &\partial_P (f_3-F_3) \sim_{\mathcal{S}_{23}^+}, \,\, \partial_P f_3 \sim_{\mathcal{S}_{23}^-}2^{1+\alphaI} \alphaI \xi_*^{-\alphaII} h^{\alphaI}\\ \nonumber  \\\nonumber 
&\partial_P (g_3-G_3) \sim_{\mathcal{S}_{23}^+}, \,\, \partial_P g_3 \sim_{\mathcal{S}_{23}^-} 2^{\alphaI} \alphaI \xi_*^{-\alphaII} h^{\alphaI-1}\\\nonumber  \\ 
\label{eqn:FS3b}
\tag{FS3b} &\partial_P g_3 >0 \text{ on } \mathcal{S}_{30}, \,\,  \partial_P f_3 \text{ on } \mathcal{S}_{30} 
\\ \nonumber \\
&\nonumber F_3 \sim_{\H_{h_0}}  \alphaII I_{f_3}(\infty) h^{\frac{1}{2}+\alphaI}, \,\, G_3 \sim_{\H_{h_0}} \alphaII I_{g_3}(\infty) h^{-\frac{1}{2}+\alphaI}\\ \nonumber \\
&\nonumber I_{f_3}(\infty) =\frac{\sqrt{2\pi}}{a^{\alphaI}} \frac{\Gamma(\alphaI)}{\Gamma(\alphaI+ \tfrac{1}{2})} \frac{\alpha}{\alpha+2}, \,\, I_{g_3}(\infty)= \frac{\sqrt{2\pi}}{2a^{\alphaI}} \frac{\Gamma(\alphaI)}{\Gamma(\alphaI+ \tfrac{1}{2})}\\ \nonumber 
\end{align}

\begin{equation}
\label{eqn:FS2a}
\tag{FS2a}
\begin{array}{lr}
|\partial_P^i f_2| \precsim_{\mathcal{S}_2^\pm}C(\xi_*) h^{\frac{1}{2}- \frac{1}{2}i+\alphaI},&\,\, |\partial_P^i g_2|\precsim_{\mathcal{S}_2^\pm}C(\xi_*) h^{-\frac{1}{2}-\frac{1}{2}i +\alphaI}\\\\
|\partial_P^i F_2| \precsim_{\H_{h_0}} C(\xi_*) h^{\frac{1}{2}- \frac{1}{2}i+\alphaI}, &\,\, |\partial_P^i G_2| \precsim_{\H_{h_0}} C(\xi_*) h^{-\frac{1}{2}-\frac{1}{2}i +\alphaI}
\end{array}
\end{equation}

\begin{align}
\nonumber \\
\nonumber &\partial_P f_2  \sim_{\mathcal{S}_{12}^-}, \,  \partial_P (f_2 -\tfrac{1}{2}F_2) \sim_{\mathcal{S}_{12}^+} -2^{1-\betaI} \betaI \xi_*^{\frac2\beta} h^{-\betaI}\\ \nonumber \\
\nonumber &\partial_P g_2 \sim_{\mathcal{S}_{12}^-}, \, \partial_P (g_2 -\tfrac{1}{2}G_2) \sim_{\mathcal{S}_{12}^+} -2^{-\betaI} \betaI \xi_*^{\frac2\beta} h^{-1-\betaI}\\   \nonumber \\ \label{eqn:FS2b} \tag{FS2b}
 &\partial_P (f_2-\tfrac{1}{2}F_2) \sim_{\mathcal{S}_{23}^-}, \,\partial_P f_2 \sim_{\mathcal{S}_{23}^+} 2^{1+\alphaI} \alphaI \xi_*^{-\alphaII} h^{\alphaI}\\\nonumber \\
\nonumber &\partial_P (g_2-\tfrac{1}{2}G_2) \sim_{\mathcal{S}_{23}^-}, \, \partial_P g_2 \sim_{\mathcal{S}_{23}^+} 2^{\alphaI} \alphaI \xi_*^{-\alphaII} h^{-1+\alphaI}\\ \nonumber \\
\nonumber &F_2\sim_{\H_{h_0}} 2^{\frac{3}{2}+\alphaI} \xi_*^{-\alphaII} h^{\frac{1}{2}+\alphaI}, \,\, G_2 \sim_{\H_{h_0}} 2^{\frac{1}{2}+\alphaI} \xi_*^{-\alphaII} h^{\alphaI-\frac{1}{2}}
\end{align}

\begin{align}
\label{eqn:AF}
\tag{AF}
\mathcal{A}_\pm(h) \sim_{\H_{h_0}}\frac{2\alpha}{\alpha+2} h, \,\,\, |\partial_P \mathcal{A}_{\pm}|\precsim_{\H_{h_0}} C(\xi_*) h^{\frac{1}{2}}, \,\, \,|\partial_P^2 \mathcal{A}_\pm| \precsim_{\H_{h_0}} C(\xi_*) .\\\nonumber
\end{align}

The formulas~\eqref{eqn:AF} follow easily from the formulas above it.  Before establishing all of the remaining formulas, first recall the definitions of $k(h)$ and $j(h)$ introduced just above Definition~\ref{def:approxdyn}.

\begin{proof}[Proof of Formulas~\eqref{eqn:FS1a} and~\eqref{eqn:FS1b}]
Recall that in the region $\mathcal{S}_1\setminus \mathcal{S}_{12}^+$, the approximate dynamics is that dynamics determined by the Hamiltonian $K=K(Q,P)$.  We will first establish a few helpful facts about this dynamics.  Letting $\eta=PQ^{\frac\beta2}$ note that we can express $K$ as follows 
\begin{align*}
K= \frac{P^2}{2}+ b Q^{-\beta}=  Q^{-\beta} \Big( \frac{\eta^2}{2} + b \Big) .  
\end{align*}  
Observe that while $K$ is Hamiltonian in $(Q,P)$, it is not Hamiltonian in $(Q, \eta)$.  Nevertheless, applying the chain rule produces
\begin{align*}
\dot{\eta} &= \beta Q^{-\frac\beta2-1} \Big( \frac{\eta^2}{2} + b\Big)\\
\dot{Q}&=  Q^{-\frac\beta2} \eta.  
\end{align*} 
and that, for any two points $(Q(\eta_1), \eta_1)$ and $(Q(\eta_2), \eta_2)$ on the same solution curve $K=k(h)$,  
\begin{align}
\label{eqn:curveq1}
Q(\eta_1) = Q(\eta_2) \frac{\Big(\frac{\eta^2_1}{2}+ b \Big)^{1/\beta}}{\Big(\frac{\eta_2^2}{2}+ b \Big)^{1/\beta}}.  
\end{align}
Using these facts, we will now derive quasi-explicit formulas for $f_1, g_1, F_1, G_1$ from which~\eqref{eqn:FS1a} and \eqref{eqn:FS1b} will follow.

Notice that for $(Q,P)\in \mathcal{S}_1 \cap \Gamma(h)$:
\begin{align}
\nonumber f_1(Q,P)&= \int_0^{\tau_1(Q,P)} P_s^2 \, ds =  \int_0^{\tau_1(Q,P)} \eta_s^2 Q_s^{-\beta} \, ds = \int_{PQ^{\frac\beta2}}^{\xi_*} \eta^2 Q(\eta)^{-\beta } \bigg| \frac{ds}{d\eta}\bigg| \,d \eta\\
\label{eqn:qstar1}&=\frac{Q(\xi_*)^{1-\frac\beta2}}{\beta\big( \frac{\xi_*^2}{2}+b\big)^{-\frac{1}{2}+ \betaI}}\int_\eta^{\xi_*} \frac{x^2}{\big( \frac{x^2}{2}+b\big)^{\frac{3}{2}-\betaI}}\, dx
\end{align}
where in the last equality $\eta=PQ^{\frac\beta2}$, we related $Q(\eta)$ with $Q(\xi_*)$ using equation~\eqref{eqn:curveq1} and we replaced the variable of integration with $x$.  Noting that 
\begin{align}
\label{eqn:curveq2}
Q(\xi_*) =  \frac{\big(\frac{\xi_*^2}{2} +b \big)^{1/\beta}}{k(h)^{1/\beta}} 
\end{align}  
and replacing $Q(\xi_*)$ in equation~\eqref{eqn:qstar1} with the righthand side of~\eqref{eqn:curveq2} produces the following formula $f_1 $ for $(Q,P) \in \mathcal{S}_2\cap \Gamma(h)$:
\begin{align*}
\boxed{f_1(Q,P) = \betaI k(h)^{\frac{1}{2}-\betaI} I_{f_1}(\eta, \xi_*) }
\end{align*}
where $\eta=PQ^{\frac\beta2}$ and 
\begin{align*}
I_{f_1}(\eta, \xi_*) = \int_\eta^{\xi_*} \frac{x^2}{\big( \frac{x^2}{2}+b\big)^{\frac{3}{2}-\betaI}}\, dx.  \end{align*}
Plugging in $\eta=-\xi_*$ using the fact that the integrand is an even function produces the following formula for $F_1$ 
\begin{align*}
\boxed{F_1(h) = \frac2\beta k(h)^{\frac{1}{2}-\betaI} I_{f_1}(0, \xi_*)}.  
\end{align*}
  
To derive a similar expression for $g_1$ and hence $G_1$, following a similar line of reasoning we notice that for $(Q,P) \in \mathcal{S}_1\cap \Gamma(h)$:
\begin{align*}
g_1(Q,P)&= \int_0^{\tau_1(Q,P)} \, ds= \int_{PQ^{\frac\beta2}}^{\xi_*} \bigg| \frac{ds}{d\eta}\bigg| \,d \eta\\ &= \frac{1}{\beta}\frac{Q(\xi_*)^{1+\frac{\beta}{2}}}{\big( \frac{\xi_*^2}{2}+ b\big)^{\hf+ \betaI}} \int_{\eta}^{\xi_*} \frac{1}{\big(\frac{x^2}{2}+ b \big)^{\frac{1}{2}-\betaI}} \, dx.
\end{align*}
Again, replacing $Q(\xi_*)$ with the righthand side of~\eqref{eqn:curveq2} we find that for $(Q,P) \in \mathcal{S}_1 \cap \Gamma(h)$
\begin{align*}
\boxed{g_1(Q,P) = \betaI k(h)^{-\frac{1}{2}-\betaI} I_{g_1}(\eta, \xi_*) }
\end{align*}
where 
\begin{align*}
I_{g_1}(\eta, \xi_*) = \int_{\eta}^{\xi_*} \frac{1}{\big(\frac{x^2}{2}+ b \big)^{\frac{1}{2}-\betaI}} \, dx.
\end{align*}
Hence, we see that 
\begin{align*}
\boxed{G_1(h)= \frac2\beta k(h)^{-\frac{1}{2}-\betaI} I_{g_1}(0, \xi_*)}.  
\end{align*}

Now we can use these boxed expressions to establish the claimed formulas.  Indeed, observe that because 
\begin{align*}
I_{f_1}(0, \xi_*) &= \beta 2^{\frac{1}{2}-\betaI}\xi_*^{\frac2\beta}(1+ o(1))  \\
I_{g_1}(0, \xi_*)&= \beta 2^{-\frac{1}{2}-\betaI} \xi_*^{\frac2\beta}(1+o(1))
\end{align*}  
as $\xi_* \rightarrow \infty$ and $f_1 \leq F_1$, $g_1 \leq G_1$ we obtain
\begin{equation*}
\begin{array}{lr}
f_1 \precsim_{\mathcal{S}_1} C(\xi_*) h^{\frac{1}{2}-\betaI} & \qquad g_1 \precsim_{\mathcal{S}_1} C(\xi_*) h^{-\frac{1}{2}-\betaI}\\
F_1 \sim_{\H_{h_0}} 2^{\frac{3}{2}-\betaI} \xi_*^{\frac2\beta} h^{\frac{1}{2}-\betaI}& \qquad G_1   \sim_{\H_{h_0}} 2^{\frac{1}{2}-\betaI} \xi_*^{\frac2\beta} h^{-\frac{1}{2}-\betaI}.  
\end{array}
\end{equation*}
Also, it is not hard to check that by differentiating the boxed formulas above
\begin{equation*}
\begin{array}{lr}
\partial_P f_1 \precsim_{\mathcal{S}_1} C(\xi_*) h^{-\betaI}& \qquad \qquad \partial_P g_1 \precsim_{\mathcal{S}_1} C(\xi_*) h^{-1-\betaI}\\
\partial_P^2 f_1 \precsim_{\mathcal{S}_1} C(\xi_*) h^{-\frac{1}{2}-\betaI}& \qquad \qquad \partial_P^2 g_1 \precsim_{\mathcal{S}_1} C(\xi_*) h^{-\frac{3}{2}-\betaI}\\
\partial_P F_1 \precsim_{\H_{h_0}} C(\xi_*) h^{-\betaI}& \qquad \qquad \partial_P G_1 \precsim_{\H_{h_0}} C(\xi_*) h^{-1-\betaI}\\
\partial_P^2 F_1 \precsim_{\H_{h_0}} C(\xi_*) h^{-\frac{1}{2}-\betaI}& \qquad \qquad \partial_P^2 G_1 \precsim_{\H_{h_0}} C(\xi_*) h^{-\frac{3}{2}-\betaI}  .
\end{array}
\end{equation*}
In order to obtain the remaining precise formulas, notice that on the relevant domain
\begin{align*}
\partial_P f_1& = P \betaI(\hf-\betaI) k(h)^{-\tfrac{1}{2}-\betaI} I_{f_1}(\eta, \xi_*)- \betaI k(h)^{\frac{1}{2}-\betaI} \frac{\eta^2 Q^{\frac\beta2}}{\big( \frac{\eta^2}{2} + b\big)^{\frac{3}{2}-\betaI}} \\
\partial_P g_1&= P \betaI(-\hf-\betaI) k(h)^{-\frac{3}{2}-\betaI} I_{g_1}(\eta, \xi_*)- \betaI k(h)^{-\frac{1}{2}-\betaI} \frac{Q^{\frac\beta2}}{\big( \frac{\eta^2}{2} + b\big)^{\frac{1}{2}-\betaI}} \\
\partial_P F_1&=   2 P \betaI(\hf-\betaI) k(h)^{-\frac{1}{2}-\betaI} k'(h) I_{f_1}(0, \xi_*)\\
\partial_P G_1&=  2 P \betaI(-\hf-\betaI) k(h)^{-\frac{3}{2}-\betaI} k'(h) I_{g_1}(0, \xi_*).
\end{align*}
Using the above, we find that since $\lim_{h\rightarrow \infty } (h^{-1}k'(h))=1$
\begin{align*}
\partial_P(f_1- F_1) \sim_{\mathcal{S}_{12}^-} - 2^{1-\betaI}\betaI \xi_*^{\frac2\beta} h^{-\betaI},& \,\,
\partial_P f_1 \sim_{\mathcal{S}_{12}^+} - 2^{1-\betaI}\betaI \xi_*^{\frac2\beta} h^{-\betaI}\\
\partial_P(g_1-G_1) \sim_{\mathcal{S}_{12}^-}- 2^{-\betaI} \xi_*^{\frac2\beta} h^{-1-\betaI},& \,\,  \partial_p g_1 \sim_{\mathcal{S}_{12}^+} - 2^{-\betaI} \xi_*^{\frac2\beta} h^{-1-\betaI}
\end{align*}
and that $\partial_P g_1 <0$ on $\mathcal{S}_{10}$ and $\partial_P f_1=0$ on $\mathcal{S}_{10}$.  Note that this now finishes the proof of the first set of formulas~\eqref{eqn:FS1a} and~\eqref{eqn:FS1b}.  
\end{proof}

Next we will derive the formulas~\eqref{eqn:FS3a} and~\eqref{eqn:FS3b}.

\begin{proof}[Proof of Formulas~\eqref{eqn:FS3a} and~\eqref{eqn:FS3b}]
In the region $\mathcal{S}_3$, we will follow a process similar to the proof of the formulas~\eqref{eqn:FS1a} and~\eqref{eqn:FS1b}.  First, express the Hamiltonian in the region $\mathcal{S}_3$ as follows:
\begin{align*}
J= \frac{P^2}{2}+ a Q^{\alpha} = Q^\alpha\Big( \frac{\xi^2}{2}+ a \Big)
\end{align*}
where $\xi=PQ^{-\frac\alpha2}$.  We will now derive some helpful facts about the dynamics along $J$.  As before with $K$, while $J$ is Hamiltonian in $(Q,P)$ it is not Hamiltonian in $(Q, \xi)$. Nonetheless,
the chain rule gives that 
\begin{align*}
\dot{\xi}&=- \alpha Q^{\frac{\alpha}{2}-1}\Big( \frac{\xi^2}{2} +a \Big)\\
\dot{Q}&= Q^{\frac\alpha2} \xi.  
\end{align*}        
Moreover, for any two points $(Q(\xi_1), \xi_1)$ and $(Q(\xi_2), \xi_2)$ on the same solution curve $J=j(h)$, we have that 
\begin{align*}
Q(\xi_1) = Q(\xi_2) \frac{\Big( \frac{\xi_2^2}{2}+ a\Big)^{1/\alpha}}{\Big( \frac{\xi^2_1}{2}+ a\Big)^{1/\alpha}}.  
\end{align*}
We now derive quasi-explicit expressions for $f_3, g_3, F_3, G_3$ from which the claimed formulas~\eqref{eqn:FS3a} and~\eqref{eqn:FS3b} will follow.

Notice that for $(Q,P) \in \mathcal{S}_3 \cap \Gamma(h)$ we have 
\begin{align}
\nonumber f_3(Q,P)&=\int_0^{\tau_3(Q,P)} P_s^2 \, ds = \int_0^{\tau_3(q,p)} \xi_s^2 Q_s^\alpha \, ds= \int_{-\xi_*}^{PQ^{-\frac\alpha2}}\xi^2 Q(\xi)^{\alpha}  \bigg|\frac{ds}{d\xi}\bigg| \, d\xi \\
\label{eqn:expref1}&= \frac{1}{\alpha}Q(\xi_*)^{\frac{\alpha}{2}+1} \Big(\frac{\xi_*^2}{2} + a \Big)^{\frac{1}{2}+ \alphaI}\int_{-\xi_*}^{\xi} \frac{x^2}{\big( \frac{x^2}{2}+a\big)^{\frac{3}{2}+ \alphaI}}\, dx
\end{align}
where in the last formula we have written $\xi=PQ^{-\frac\alpha2}$.  
Noting that 
\begin{align}
\label{eqn:qexpr3}
Q(\xi_*)= \frac{j( h)^{1/\alpha}}{\Big(\frac{\xi_*^2}{2}+a \Big)^{1/\alpha}}
\end{align}
we can substitute this into~\eqref{eqn:expref1} to find that for $(Q,P) \in \mathcal{S}_3 \cap \Gamma(h)$:
\begin{align*}
\boxed{f_3(Q,P)= \alphaI j(h)^{\frac{1}{2}+\alphaI} I_{f_3}(-\xi_*, \xi)}
\end{align*}
where
\begin{align*}
I_{f_3}(-\xi_*, \xi)= \int_{-\xi_*}^{\xi} \frac{x^2}{\big( \frac{x^2}{2}+a\big)^{\frac{3}{2}+ \alphaI}}\, dx.
\end{align*}
By plugging $\xi=\xi_*$ into the formula for $f_3$ and using the fact that the integrand above is even, we see that 
\begin{align*}
\boxed{F_3(h)= \alphaII j(h)^{\frac{1}{2}+\alphaI} I_{f_3}(0, \xi_*)}.
\end{align*}

To obtain a quasi-explicit formula for $g_3(Q,P)$, follow a similar line of reasoning to find that for $(Q,P) \in \mathcal{S}_3\cap \Gamma(h)$
\begin{align}
\nonumber g_3(Q,P) = \int_0^{\tau_3(Q,P)} \, ds &=\alphaI \int_{-\xi_*}^{PQ^{-\frac\alpha2}}  \frac{Q(\xi)^{1-\frac\alpha2}}{ \big(\frac{\xi^2}{2}+a \big)} \, d\xi\\
\label{eqn:g3expr}&=  \alphaI\frac{Q(\xi_*)^{1-\frac{\alpha}{2}}}{ \big(\frac{\xi^2_*}{2} +a\big)^{\frac{1}{2}-\alphaI}}\int_{-\xi_*}^{\xi} \frac{1}{\big( \frac{x^2}{2}+a\big)^{\frac{1}{2}+ \alphaI}} \, dx.  \end{align}  
Substituting the righthand side of~\eqref{eqn:qexpr3} in for $Q(\xi_*)$ we see that for $(Q,P) \in \mathcal{S}_3\cap \Gamma(h)$
\begin{align*}
\boxed{g_3(Q,P) = \alphaI j(h)^{-\frac{1}{2}+\alphaI} I_{g_3}(-\xi_*, \xi)}
\end{align*}
where
\begin{align*}
I_{g_3}(-\xi_*, \xi)= \int_{-\xi_*}^{\xi} \frac{1}{\big( \frac{x^2}{2}+a\big)^{\frac{1}{2}+ \alphaI}} \, dx.\end{align*}
Substituting $\xi=\xi_*$ into the formula for $g_3$ produces the following expression for $G_3$
\begin{align*}
\boxed{G_3(h)= \alphaII j(h)^{-\frac{1}{2}+\alphaI} I_{g_3}(0, \xi_*)}.
\end{align*}

By using and differentiating the boxed formulas above, we can easily see that 
\begin{equation*}
\begin{array}{lr}
f_3 \precsim_{\mathcal{S}_1} C(\xi_*) h^{\frac{1}{2}+\alphaI},&\,\, g_3  \precsim_{\mathcal{S}_1} C(\xi_*) h^{-\frac{1}{2}+\alphaI}\\
\partial_P f_3 \precsim_{\mathcal{S}_1} C(\xi_*) h^{\alphaI},&\,\, \partial_P g_3  \precsim_{\mathcal{S}_1} C(\xi_*) h^{-1+\alphaI}\\
\partial_P^2 f_3  \precsim_{\mathcal{S}_1} C(\xi_*) h^{-\frac{1}{2}+\alphaI},&\,\, \partial_P^2 g_3  \precsim_{\mathcal{S}_1} C(\xi_*) h^{-\frac{3}{2}+\alphaI}\\
\partial_P F_3 \precsim_{\H_{h_0}} C(\xi_*) h^{\alphaI}, &\,\,\partial_P G_3 \precsim_{\H_{h_0}} C(\xi_*) h^{-1+\alphaI}\\
\partial_P^2 F_3 \precsim_{\H_{h_0}} C(\xi_*) h^{-\frac{1}{2}+\alphaI}, &\,\,\partial_P G_3 \precsim_{\H_{h_0}} C(\xi_*) h^{-\frac{3}{2}+\alphaI}.  
\end{array}
\end{equation*}
To arrive at the precise formulas, we need the following:

\vspace{0.1in}

\emph{Claim 1}.  
\begin{align*}
I_{f_3}(\infty)=\lim_{\xi_*\rightarrow \infty}I_{f_3}(0, \xi_*) &=\frac{\sqrt{2\pi}}{a^{\alphaI}} \frac{\Gamma(\alphaI)}{\Gamma(\alphaI+ \tfrac{1}{2})} \frac{\alpha}{\alpha+2}\\
I_{g_3}(\infty)=\lim_{\xi_*\rightarrow \infty}I_{g_3}(0, \xi_*)&=\frac{\sqrt{2\pi}}{2a^{\alphaI}} \frac{\Gamma(\alphaI)}{\Gamma(\alphaI+ \tfrac{1}{2})} 
\end{align*}

\vspace{0.1in}

\emph{Proof of Claim 1}.  
This fact follows easily from the formula
\begin{align*}
\int_0^\infty \frac{1}{(x^2 +1)^p} \, dx = \frac{\sqrt{\pi}}{2} \frac{\Gamma(p-\frac{1}{2})}{\Gamma(p)},  
\end{align*}
which is valid for $p>\frac{1}{2}$, and by basic integral substitution methods.  $\blacksquare$   

\vspace{0.1in}

Along with the formulas:
\begin{align*}
\partial_P f_3& = \alphaI( \tfrac{1}{2}+ \alphaI) j(h)^{-\frac{1}{2}+\alphaI} I_{f_3}(- \xi_*, \xi) P+ \alphaI j(h)^{\frac{1}{2}+\alphaI} \frac{\xi^2 Q^{-\frac\alpha2}}{\big(\frac{\xi^2}{2} +a \big)^{\frac{3}{2}+\alphaI}}\\
\partial_P g_3& = \alphaI( \alphaI-\tfrac{1}{2}) j(h)^{\alphaI-\frac{3}{2}} I_{g_3}(- \xi_*, \xi) P+ \alphaI j(h)^{\alphaI-\frac{1}{2}} \frac{Q^{-\frac\alpha2}}{\big(\frac{\xi^2}{2} +a \big)^{\frac{1}{2}+\alphaI}}\\
\partial_P F_3&= \alphaI(\tfrac{1}{2}+\alphaI) j(h)^{\alphaI-\frac{1}{2}}j'(h) I_{f_3}(-\xi_*, \xi_*) P\\
\partial_P G_3&= \alphaI(\alphaI-\tfrac{1}{2}) j(h)^{\alphaI-\frac{3}{2}}j'(h) I_{g_3}(-\xi_*, \xi_*) P,\end{align*}
and the fact that $h^{-1}j'(h) \rightarrow 1 $ as $h\rightarrow \infty$, Claim 1 now allows us to conclude the following:
\begin{equation*}
\begin{array}{lr}
F_3 \sim_{\H_{h_0}} \alphaII I_{f_3}(\infty) h^{\frac{1}{2}+\alphaI}, & G_3 \sim_{\H_{h_0}} \alphaII I_{g_3}(\infty) h^{-\frac{1}{2}+\alphaI}\\\\
\partial_Pf_3\sim_{\mathcal{S}_{23}^-} 2^{1+\alphaI} \alphaI \xi_*^{-\alphaII} h^{\alphaI},& \partial_P(f_3-F_3)\sim_{\mathcal{S}_{23}^+} 2^{1+\alphaI} \alphaI \xi_*^{-\alphaII} h^{\alphaI}  \\\\
\partial_P g_3 \sim_{\mathcal{S}_{23}^-}   2^{\alphaI} \alphaI \xi_*^{-\alphaII} h^{\alphaI-1}, &  \partial_P (g_3-G_3) \sim_{\mathcal{S}_{23}^+}   2^{\alphaI} \alphaI \xi_*^{-\alphaII} h^{\alphaI-1}.\end{array}
\end{equation*}
Moreover, using the expressions above we find that $\partial_P g_3 >0$ on $\mathcal{S}_{30}$ and that $\partial_P f_3=0$ on $\mathcal{S}_{30}$.  Note that this finishes the proof of formulas~\eqref{eqn:FS3a} and~\eqref{eqn:FS3b}. 
\end{proof}



We now establish the remaining formulas~\eqref{eqn:FS2a} and~\eqref{eqn:FS2b}

\begin{proof}[Proof of Formulas~\eqref{eqn:FS2a} and~\eqref{eqn:FS2b}]

In the region $\mathcal{S}_2$, we will use the coordinates $(Q,h)$ where 
\begin{align*}
h= \frac{P^2}{2}+U(Q). 
\end{align*} 
To start, recall the quantities $Q_1=Q_1(\xi_*, h)$ and $Q_3=Q_3(\xi_*, h)$ introduced just above Definition~\ref{def:approxdyn}.   
Both of the quantities $Q_1$ and $Q_3$ exist and are twice continuously differentiable in $h$ for $h\geq h_0$ for all $h_0=h_0(\xi_*)$ large enough.  These derivatives will be denoted by $Q_i'$ and $Q_i''$ below.  Now observe that \begin{align*}
Q_3^\alpha \sim_{\H_{h_0}} \frac{2 h}{\xi^2_*}\,\, &\text{ and }\,\, Q_1^{-\beta} \sim_{\H_{h_0}} \frac{2 h}{\xi^2_*}\\
Q_3' \sim_{\H_{h_0}} \alphaI2^{\alphaI} \xi^{-\alphaII}_* h^{-1+\alphaI} \,\, &\text{ and } \,\, Q_1' \sim_{\H_{h_0}} - \betaI 2^{-\betaI} \xi_*^{\frac2\beta} h^{-1-\betaI}\\
|Q_3''|\precsim_{\H_{h_0}} C(\xi_*) h^{-2 + \alphaI}\,\, &\text{ and } \,\, |Q_1''| \precsim_{\H_{h_0}} C(\xi_*) h^{-2-\betaI}.
\end{align*}  
for some constant $C(\xi_*)$ depending on $\xi_*$.  

We now derive the formulas for $f_2, g_2, F_2, G_2$ from which the claimed formulas will follow.  For notational purposes, let 
\begin{align*}
I(a,b, \zeta) = \int_a^b (h- U(q))^{\zeta}\, dq.    
\end{align*}
After changing variables from $s$ to $Q$ we find that on $\mathcal{S}_2^-$
\begin{align*}
\Aboxed{f_2(Q,P)&= \int_0^{\tau_2(Q,P)} P_s^2 \,ds =\int_{Q_1}^{Q} \sqrt{2}(h-U(q))^{\hf} \, dq =\sqrt{2} I(Q_1, Q, \tfrac{1}{2})}\\
\Aboxed{g_2(Q,P)&= \int_0^{\tau_2(Q,P)} \, ds = \int_{Q_1}^Q \frac{1}{\sqrt{2}} (h-U(q))^{-\hf} \, dq  = \frac{1}{\sqrt{2}} I(Q_1, Q, -\tfrac{1}{2})}
\end{align*}
and on $\mathcal{S}_2^+$
\begin{align*}
\Aboxed{f_2(Q,P)&= \int_0^{\tau_2(Q,P)} P_s^2 \, ds = \int_Q^{Q_3} \sqrt{2} (h- U(q))^{\hf} \, dq = \sqrt{2} I(Q, Q_3, \tfrac{1}{2})}\\
\Aboxed{g_2(Q,P)&= \int_0^{\tau_2(Q,P)}  \, ds = \int_{Q}^{Q_3} \frac{1}{\sqrt{2}} (h-U(q))^{-\hf} \, dq  =\frac{1}{\sqrt{2}} I(Q, Q_3, -\tfrac{1}{2}) }
\end{align*}
It is important to remark that each of the quantities above is twice continuously differentiable in $Q$, $P$ and $h$ on their respective domains, as $|P|=\sqrt{2}(h-U(Q))^{\hf}$ is bounded below on $\mathcal{S}_{2}$ by $\xi_*/Q^{\beta/2}$.  The following expressions 
\begin{align*}
\Aboxed{F_2(h)&= 2^{\frac{3}{2}} I(Q_1, Q_3, \tfrac{1}{2})}\\
\Aboxed{G_2(h)&= 2^{\frac{1}{2}}I(Q_1, Q_3, -\tfrac{1}{2})}
\end{align*}    
follow easily by substituting the relevant endpoint, either $Q=Q_1$ or $Q=Q_3$, into the formulas above and doubling the result via symmetry.  To obtain the desired formulas, we will need the following claim.

\emph{Claim 2}. 
\begin{align*}
I(Q_1, Q_3, \zeta) \sim_{\H_{h_0}}2^{\alphaI} \xi_*^{-\alphaII} h^{\zeta + \alphaI}
\end{align*}

\emph{Proof of Claim 2}.  Consider the modified potential 
\begin{align*}
\widetilde{U}_\epsilon(q) \eqdef a q^\alpha + \sum_{i=2}^l a_i \epsilon^{r_i} q^{\alpha_i}, \,\, r_i = 1-\alpha_i/\alpha,
\end{align*}     
which has the scaling property $\widetilde{U}_\epsilon ( t^{\alphaI} q) = t \widetilde{U}_{\epsilon/t}(q) $ for $\epsilon, q,t >0$.  Since $h>0$ is constant in the integral $I(Q_1, Q_3, \zeta)$ we note that    
\begin{align*}
I(Q_1, Q_1, \zeta) &=h^{\zeta} \int_{Q_1}^{Q_3} (1- h^{-1} \widetilde{U}_1(q)) \, dq\\
& = h^{\zeta} \int_{Q_1}^{Q_3}  (1- \widetilde{U}_{h^{-1}}( h^{-\frac{1}{\alpha}} q) )\, dq \\
&= h^{\zeta+ \alphaI} \int_{h^{-\frac{1}{\alpha}}Q_1}^{h^{-\frac{1}{\alpha}}Q_3}  (1- \widetilde{U}_{h^{-1}}(q) )^\zeta \, dq
\end{align*}  
where in the last equality we made an integral substitution.  
Observe that for $q\in [Q_1 h^{-\frac{1}{\alpha}}, Q_3 h^{-\frac{1}{\alpha}}]$ we have the bounds for $h\geq h_0$:
\begin{align*}
1- U_{h^{-1}}(q)& \leq 1 + \sum_{i: \alpha_i \geq 0} |a_i| h^{-r_i} (Q_3 h^{-\alphaI})^{\alpha_i}+ \sum_{i: \alpha_i <0} |a_i| h^{-r_i} (Q_1 h^{-\alphaI})^{\alpha_i}\\
&= 1+  \sum_{i: \alpha_i \geq 0} |a_i| h^{-1} Q_3^{\alpha_i}+ \sum_{i: \alpha_i <0} |a_i| h^{-1} Q_1^{\alpha_i},\\
1- U_{h^{-1}}(q) &\geq 1 - \sum_{i: \alpha_i \geq 0} |a_i| h^{-r_i} (Q_3 h^{-\alphaI})^{\alpha_i}- \sum_{i: \alpha_i <0} |a_i| h^{-r_i} (Q_1 h^{-\alphaI})^{\alpha_i}\\
&=   1-  \sum_{i: \alpha_i \geq 0} |a_i| h^{-1} Q_3^{\alpha_i}- \sum_{i: \alpha_i <0} |a_i| h^{-1} Q_1^{\alpha_i}.  
\end{align*} 
Applying the asymptotic formulas for $Q_i$, $i=1,3$, it follows from the above bounds that for every $\epsilon >0$, there exists $\xi_*>0$ such that for all $h$ large enough we have 
\begin{align*}
(1- \epsilon) h^{\zeta + \alphaI} h^{-\alphaI}( Q_3 - Q_1) \leq I(Q_1, Q_1, \zeta) \leq (1+ \epsilon) h^{\zeta + \alphaI} h^{-\alphaI} (Q_3 - Q_1).
\end{align*}
Using the asymptotic formulas for $Q_1$ and $Q_3$ again, we obtain the claimed formula.  $\blacksquare$ 

\vspace{0.1in}

Using these quasi-explicit expressions for $f_2, g_2, F_2, G_2$, Claim~2 and the asymptotic formulas for $Q_1, Q_3$ and their derivatives, it is not hard to show that 
\begin{align*}
\begin{array}{lr}
f_2 \precsim_{\mathcal{S}_2^\pm} C(\xi_*) h^{\frac{1}{2}+\alphaI}, & g_2 \precsim_{\mathcal{S}_2^\pm} C(\xi_*) h^{\alphaI-\frac{1}{2}}\\
\partial_P f_2 \precsim_{\mathcal{S}_2^\pm} C(\xi_*) h^{\alphaI}, & \partial_P g_2 \precsim_{\mathcal{S}_2^\pm} C(\xi_*) h^{\alphaI-1} \\
\partial_P^2 f_2 \precsim_{\mathcal{S}_2^\pm} C(\xi_*) h^{\alphaI-\frac{1}{2}}, & \partial_P^2 g_2 \precsim_{\mathcal{S}_2^\pm} C(\xi_*) h^{\alphaI-\frac{3}{2}}\\
F_2 \sim_{\H_{h_0}} 2^{\frac{3}{2}+\alphaI} \xi_*^{-\alphaII} h^{\frac{1}{2}+\alphaI}, &\,\, G_2 \sim_{\H_{h_0}} 2^{\alphaI+\frac{1}{2}} \xi_*^{-\alphaII} h^{\alphaI-\frac{1}{2}}\\
\partial_P F_2 \precsim_{\H_{h_0}} C(\xi_*) h^{\alphaI}, & \partial_P G_2 \precsim_{\H_{h_0}} C(\xi_*) h^{\alphaI-1} \\
\partial_P^2 F_2 \precsim_{\H_{h_0}} C(\xi_*) h^{\alphaI-\frac{1}{2}}, & \partial_P^2 G_2 \precsim_{\H_{h_0}} C(\xi_*) h^{\alphaI-\frac{3}{2}}
\end{array}
\end{align*}
To obtain the precise formulas, observe that on $\mathcal{S}_2^-$ 
\begin{align*}
\partial_P f_2(Q,P) &= P \bigg[ -\xi_*\frac{Q_1'}{Q_1^{\frac\beta2}} + \frac{1}{\sqrt{2}}I(Q_1, Q, -\tfrac{1}{2})  \bigg] \\
\partial_P g_2(Q,P)& = P \bigg[ -\frac{Q_1^{\frac\beta2} Q_1'}{\xi_*} - \frac{1}{2^{\frac32}}I(Q_1, Q, -\tfrac{3}{2})  \bigg] 
\end{align*}
and on $\mathcal{S}_2^+$
\begin{align*}
\partial_P f_2(Q,P) &= P \bigg[ \xi_* Q_3^{\frac\alpha2}Q_3'+ \frac{1}{\sqrt{2}}I(Q, Q_3, -\tfrac{1}{2})  \bigg] \\
\partial_P g_2(Q,P)& = P \bigg[ \frac{Q_3^{-\frac\alpha2} Q_3'}{\xi_*} - \frac{1}{2^{\frac32}}I(Q, Q_3, -\tfrac{3}{2})  \bigg].  
\end{align*}  
Also realize that 
\begin{align*}
\partial_P (F_2(h)) &= 2 P \bigg[-\xi_*\frac{Q_1'}{Q_1^{\frac\beta2}} +  \xi_* Q_3^{\frac\alpha2}Q_3' + \frac{1}{\sqrt{2}} I(Q_1, Q_3, -\tfrac{1}{2})\bigg]\\
\partial_P (G_2(h))&= 2 P \bigg[- \frac{Q_1^{\frac\beta2} Q_1'}{\xi_*}+      \frac{Q_3^{-\frac\alpha2}}{\xi_* }Q_3' - \frac{1}{2^{\frac32}}I(Q_1, Q_3, -\tfrac{3}{2})\bigg].
\end{align*} 
By plugging in the asymptotic value of $P$ on each boundary, these expressions allow us to arrive at the claimed precise asymptotic formulas
\begin{align*}
&\partial_P f_2 \sim_{\mathcal{S}_{12}^-}, \,  \partial_P (f_2 -\tfrac{1}{2}F_2) \sim_{\mathcal{S}_{12}^+} -2^{1-\betaI} \betaI \xi_*^{\frac2\beta} h^{-\betaI}\\
&\partial_P g_2 \sim_{\mathcal{S}_{12}^-}, \, \partial_P (g_2 -\tfrac{1}{2}G_2) \sim_{\mathcal{S}_{12}^+} -2^{-\betaI} \betaI \xi_*^{\frac2\beta} h^{-1-\betaI}\\
&\partial_P (f_2-\tfrac{1}{2}F_2) \sim_{\mathcal{S}_{23}^-}, \,\partial_P f_2 \sim_{\mathcal{S}_{23}^+} 2^{1+\alphaI} \alphaI \xi_*^{-\alphaII} h^{\alphaI}\\
&\partial_P (g_2-\tfrac{1}{2}G_2) \sim_{\mathcal{S}_{23}^-}, \, \partial_P g_2 \sim_{\mathcal{S}_{23}^+} 2^{\alphaI} \alphaI \xi_*^{-\alphaII} h^{-1+\alphaI},\end{align*}
finishing the proof.

\end{proof}
\end{document}